\documentclass[12pt]{amsart}
\usepackage{amssymb,latexsym}
\usepackage{pdfsync}
\usepackage{color}

\usepackage{esint}

\newdimen\AAdi%
\newbox\AAbo%
\def\AAk#1#2{\s_etbox\AAbo=\hbox{#2}\AAdi=\wd\AAbo\kern#1\AAdi{}}%
\def\AAr#1#2#3{\s_etbox\AAbo=\hbox{#2}\AAdi=\ht\AAbo\raise#1\AAdi\hbox{#3}}%
\font\tenmsb=msbm10 at 12pt
\font\sevenmsb=msbm7 at 8pt
\font\fivemsb=msbm5 at 6pt
\newfam\msbfam
\textfont\msbfam=\tenmsb
\scriptfont\msbfam=\sevenmsb
\scriptscriptfont\msbfam=\fivemsb
\def\Bbb#1{{\tenmsb\fam\msbfam#1}}

\newcommand{\beq}{\begin{equation}}
\newcommand{\eeq}{\end{equation}}
\newcommand{\beqr}{\begin{eqnarray}}
\newcommand{\eeqr}{\end{eqnarray}}
\newcommand{\ba}{\begin{array}}
\newcommand{\ea}{\end{array}}

\begin{document}

\newtheorem{thm}{Theorem}[section]
\newtheorem{lem}{Lemma}[section]
\newtheorem{cor}{Corollary}[section]
\newtheorem{rem}{Remark}
\newtheorem{pro}{Proposition}[section]
\newtheorem{defi}{Definition}
\newtheorem{eg}{Example}
\newtheorem*{claim}{Claim}
\newtheorem{conj}[thm]{Conjecture}
\newcommand{\noi}{\noindent}
\newcommand{\dis}{\displaystyle}
\newcommand{\mint}{-\!\!\!\!\!\!\int}
\numberwithin{equation}{section}

\def \bx{\hspace{2.5mm}\rule{2.5mm}{2.5mm}}
\def \vs{\vspace*{0.2cm}}
\def\hs{\hspace*{0.6cm}}
\def \ds{\displaystyle}
\def \p{\partial}
\def \O{\Omega}
\def \o{\omega}
\def \b{\beta}
\def \m{\mu}
\def \l{\lambda}
\def\L{\Lambda}
\def \ul{u_\lambda}
\def \D{\Delta}
\def \d{\delta}
\def \k{\kappa}
\def \s{\sigma}
\def \e{\varepsilon}
\def \a{\alpha}
\def \sm{\sigma}
\def \tf{\tilde{f}}
\def\cqfd{%
\mbox{ }%
\nolinebreak%
\hfill%
\rule{2mm} {2mm}%
\medbreak%
\par%
}
\def \pr {\noindent {\it Proof.} }
\def \rmk {\noindent {\it Remark} }
\def \esp {\hspace{4mm}}
\def \dsp {\hspace{2mm}}
\def \ssp {\hspace{1mm}}

\def\la{\langle}\def\ra{\rangle}

\def \u{u_+^{p^*}}
\def \ui{(u_+)^{p^*+1}}
\def \ul{(u^k)_+^{p^*}}
\def \energy{\int_{\R^n}\u }
\def \sk{\s_k}
\def \mo{\mu_k}
\def\cal{\mathcal}
\def \I{{\cal I}}
\def \J{{\cal J}}
\def \K{{\cal K}}
\def \OM{\overline{M}}

\def\n{\nabla}

\def\fk{{{\cal F}}_k}
\def\M1{{{\cal M}}_1}
\def\Fk{{\cal F}_k}
\def\Fl{{\cal F}_l}
\def\FF{\cal F}
\def\Gk{{\Gamma_k^+}}
\def\n{\nabla}
\def\uuu{{\n ^2 u+du\otimes du-\frac {|\n u|^2} 2 g_0+S_{g_0}}}
\def\uuug{{\n ^2 u+du\otimes du-\frac {|\n u|^2} 2 g+S_{g}}}
\def\sku{\sk\left(\uuu\right)}
\def\qed{\cqfd}
\def\vvv{{\frac{\n ^2 v} v -\frac {|\n v|^2} {2v^2} g_0+S_{g_0}}}
\def\vvs{{\frac{\n ^2 \tilde v} {\tilde v}
 -\frac {|\n \tilde v|^2} {2\tilde v^2} g_{S^n}+S_{g_{S^n}}}}
\def\skv{\sk\left(\vvv\right)}
\def\tr{\hbox{tr}}
\def\pO{\partial \Omega}
\def\dist{\hbox{dist}}
\def\RR{\Bbb R}\def\R{\Bbb R}
\def\C{\Bbb C}
\def\B{\Bbb B}
\def\N{\Bbb N}
\def\Q{\Bbb Q}
\def\Z{\Bbb Z}
\def\PP{\Bbb P}
\def\EE{\Bbb E}
\def\F{\Bbb F}
\def\G{\Bbb G}
\def\H{\Bbb H}
\def\SS{\Bbb S}\def\S{\Bbb S}

\def\div{\hbox{div}\,}

\def\lcf{{locally conformally flat} }

\def\circledwedge{\setbox0=\hbox{$\bigcirc$}\relax \mathbin {\hbox
to0pt{\raise.5pt\hbox to\wd0{\hfil $\wedge$\hfil}\hss}\box0 }}

\def\sss{\frac{\s_2}{\s_1}}

\date{\today}

\title[ ancient caloric functions ]{ Space of ancient caloric functions on some manifolds beyond volume doubling}

\author{} 

\author[Lin]{Fanghua Lin}
\address{Courant Institute\\ 251 Mercer St.\\New York, NY 10012 USA
 }
 \email{linf@cims.nyu.edu}

\author[Qiu]{Hongbing Qiu}
\address{School of Mathematics and Statistics\\ Wuhan University\\Wuhan 430072,
China
 }
 \email{hbqiu@whu.edu.cn}

 \author[Sun]{Jun Sun}
\address{School of Mathematics and Statistics\\ Wuhan University\\Wuhan 430072,
China
 }
 \email{sunjun@whu.edu.cn}

 \author[Zhang]{Qi S. Zhang}
\address{Department of Mathematics  \\ University of California, Riverside\\Riverside, CA 92521
USA
 }
 \email{qizhang@math.ucr.edu}

\begin{abstract}
Under a condition that breaks the volume doubling barrier, we obtain a time polynomial structure result on the space of ancient caloric functions with polynomial growth on manifolds.
 As a byproduct, it is shown that the finiteness result for the space of harmonic functions with polynomial growth on manifolds in \cite{CM97} and \cite{Li97} are essentially sharp,  except for the multi-end cases,  addressing an issue raised in \cite{CM98}  and removing all {\it local} topological or geometric conditions on the manifold with respect to a reference point.

\vskip12pt



\end{abstract}
\maketitle
\section{Introduction}

\allowdisplaybreaks

\allowdisplaybreaks

Let $M$ be a complete, noncompact $n$-dimensional Riemannian manifold.
The study of the space of harmonic and caloric functions of polynomial growth on $M$ has been active over decades. A brief summary of these research is hereby presented.
The standard assumptions on $M$ are the volume doubling condition and the mean value inequality for solutions of the Laplace or the heat equation. Under these assumptions, Colding and Minicozzi (\cite{CM97})  proved that the space of harmonic functions of a fixed polynomial growth is finite dimensional, answering a  question of Yau (\cite{Yau}),  originally asked for the more special manifolds of nonnegative Ricci curvatures.  See also \cite{Li97} by P. Li. In \cite{LW00}, and \cite{LW01} Li and Wang further treated this question for divergence and respectively, non-divergence form uniformly elliptic operators with Lipschitz coefficients in $\mathbb{R}^n$.  In \cite{HL99}, Han and Lin extended the finite dimensionality to solutions of certain elliptic equations in a strip, allowing exponential growth of fixed order. These works extended and added on earlier results by  Avellaneda and Lin \cite{AL89},  Li-Tam \cite{LT89},  Moser-Struwe \cite{MS92} and also Lin \cite{Ln96}, Colding and Minicozzi \cite{CMjdg97} where this problem under various settings, including periodic ones, were studied by a variety of methods such as homogenization and $\Gamma$ convergence. The growth of harmonic functions near infinity is also related to the boundary behavior of Riemann mapping as demonstrated in \cite{LN85}.
On complex manifolds, it was also asked in by Yau in \cite{Yau} whether or not the dimension of the spaces of holomorphic functions of a polynomial growth is bounded from above by the dimension of the corresponding spaces of polynomials on the complex Euclidean space. In this regard, solutions have been constructed by L. Ni \cite{Ni04} and Chen-Fu-Yin-Zhu \cite{CFYZ}.  In addition, this is a step in the uniformization of a complete K\"ahler manifold $M$ with positive bisectional curvature by Siu and Yau. Extensions of the finite dimensionality result has also been made to the discrete setting in Kleiner (\cite{Kle10}), who  used a similar idea to  give a new proof of a fundamental result in geometric group theory by Gromov \cite{Gro}. See also \cite{HJL15} by Hua, Jost and Liu for graphs.

For the case of the heat equation,
 in \cite{LZ19}, it is proven by Lin and Zhang that ancient caloric functions of polynomial growth of a fixed degree $d$ are polynomials of time. It is also shown that  $\mathbf{K}^d$, the space of such functions,  is finite dimensional and a non-sharp estimate of the dimension is given using a method inspired by \cite{CM97} and \cite{Li97}. See also discussions at the end of the paper.  In \cite{CM21} based on the time polynomial result in \cite{LZ19}, a sharp dimension estimate  for  $\mathbf{K}^d$ is derived.
Colding and Minicozzi also derived a sharp dimension estimate for the space of polynomial growth caloric functions on a solution to the ancient mean curvature flow with a bounded entropy (\cite{CM20}). As an application, they showed that each slice of the ancient mean curvature flow was contained in some Euclidean space with dimension depending only on the dimension of the submanifold and the upper bound of the entropy, while independent of the codimension of the flow. Similar results are proved for self-shrinkers of the mean curvature flow. On the other hand, Calle (\cite{Ca06}) proved that each slice of the ancient mean curvature flow was contained in an affine subspace with dimension bounded in terms of the density and the dimension of the evolving submanifold.

  It is natural and desirable to extend these results beyond the typical volume doubling property, in either positive or negative (by counter-examples) directions. For the space of harmonic functions, such an issue was raised in \cite{CM98} p117 and p118, line 2.  Let us also mention two other examples. One is the work of Anderson \cite{An83}, Anderson-Schoen \cite{AS85} and D. Sullivan \cite{Su83} on non-trivial bounded harmonic functions in hyperbolic manifolds. The other is the extension of some singular integrals to the non-homogeneous spaces by Nazarov, Treil and Volberg \cite{NTV03}.

The volume doubling condition, however, is used in  all the above mentioned finite dimensionality results, and it is necessary in general. For example $\mathbf{K}^d$ is infinite dimensional on the standard hyperbolic space $\mathbb H^n$ since even the space of bounded harmonic functions is of infinite dimensional, c.f. \cite{An83}, \cite{Su83}.
Nevertheless, in this paper, we manage to break the volume doubling barrier and prove the time polynomial structure result of $\mathbf{K}^d$ under much more relaxed conditions: for  any given positive parameters $\a, \e, \d, m$ such that $\e>\frac{27}{2}\d+14$,
\begin{equation}\label{eqn-BA}
\begin{cases}\aligned
1. \quad   \frac{|B(x, r_2)|}{|B(x, r_1)|} \leq & C(1+r_2)^\a \cdot e^{\frac{1}{4+\e}\cdot \frac{r_2^{2}}{r_1^{2}}},\quad  \forall x \in M, \quad r_2 \ge r_1>0;\\
2. \quad G(x, t, y)\leq & \frac{C(1+t)^m e^{-\frac{d^2(x, y)}{(4+\d)t}}}{\sqrt{|B(x, \sqrt{t})|}\cdot \sqrt{|B(y, \sqrt{t})|}}, \quad \forall x, y \in M, \quad t>0.
\endaligned
\end{cases}
\end{equation} Here $G=G(x, t, y)$ is the heat kernel on $M$ and $B(x, r)$ is the geodesic ball centered at $x$ with radius $r$ and $|B(x, r)|$ is the volume of $B(x, r)$.
It is known from  \cite {CG98} and \cite{LW99} that under volume doubling condition ($\a=0$ here), the Gaussian upper bound ($m=0$ here) is equivalent to the mean value inequality for the heat equation.  See also the remark after Theorem \ref{thm-2}. As mentioned, volume doubling and mean value inequality together is the current standard condition for the finite dimensionality results. So Condition \eqref{eqn-BA} is much more general and is amount to extending the standard condition by any polynomial weight.

 It is also well known that the Gaussian upper bound with $m=0$ holds for a wide class of manifolds without the volume doubling condition. For example on the hyperbolic space $\mathbb H^n$ and, more generally on manifolds with Ricci curvature bounded from below by a negative constant and with the bottom of the $L^2$ spectrum being positive. See Theorem 13.4 in \cite{Lib} e.g. Our extra exponent $m$ only makes the condition even broader.
 In fact by Theorem 1.1 and example (1.6) in \cite{BCG01}, if the manifold has local bounded geometry in the sense that the  injectivity radius is bounded from below by a positive constant and the Ricci curvature is bounded from below by a negative constant, then the heat kernel upper bound with suitable $m$ in Condition (\ref{eqn-BA}) follows merely from a polynomial  bounds of the volume of large geodesic balls $C r^{D_2} \ge |B(x, r)| \ge c r^{D_1}$, $\forall r \ge 1$, and some constants $D_2>D_1>0$. Indeed, that theorem implies the on diagonal upper bound $G(x, t, x) \le C t^{-D_1/(D_1+1)}, \, t \ge 1$. See also Theorem 3.1 in the same paper where conditions can be relaxed further to replace the injectivity radius and Ricci curvature lower bound by local volume doubling and local Poincar\'e inequality.  The full upper bound for the heat kernel with sufficiently large $m$ follows automatically from the general Theorem 1.1 in \cite{Gri97}. So our condition on the heat kernel essentially cover the vast majority of typical manifolds of polynomial volume growth.

  The following is a concrete example of non-doubling manifold satisfying Condition (\ref{eqn-BA}).
Let $M$ be the connected sum of $\mathbb{R}^2$ with a half cylinder, i.e.   $\mathbb{R}^2$ with a hole where the semi-infinite cylinder $S^1 \times \mathbb{R}^+$ is attached along the boundary of the hole and smoothed out. In \cite{GIS} the authors proved, among other results,  the upper bounds on the heat kernel:
\[
G(x, t, y) \le C\frac{\ln(e+t)}{t}  e^{-c d(x, y)^2/t},  \forall t>0,
\]which is qualitatively sharp since a comparable lower bound with perhaps different constants was also proven there. c.f. example 2.11 (Figure 2) there. It is clear that this manifold is not volume doubling in large scale. However it is covered by Condition \eqref{eqn-BA} with $\alpha=1$ and any $m>0$. Many other classes of manifolds for which Condition (\ref{eqn-BA}) holds can be found in the same paper.

 The following theorem and Theorem \ref{thm-ys} below are the main results of the paper.

\begin{thm}
\label{thm-2}

Let $M$ be a complete $n$-dimensional noncompact Riemannian manifold satisfying Condition (\ref{eqn-BA}), with a reference point $0$  and  $\e>\frac{27}{2}\d+14$.

(a). Let $u$ be an ancient solution to the heat equation (\ref{eqn-H1}) with polynomial growth of degree at most $d$, namely, for any $(x, t)\in M \times (-\infty, 0]$
\begin{equation}\label{eqn-HA26}\aligned
\left| u(x, t) \right| \leq C\left( 1+d^2(x, 0) + |t| \right)^{\frac{d}{2}}.
\endaligned
\end{equation}
Let $k$ be the least integer greater than $2m+  \frac{5}{2}\a +\frac{d}{2}$. Then $u$ is a polynomial of $t$ of degree of at most $k$.

(b).  Let $\mathbf{K}^d$ be the space of ancient solutions in part (a). Then $\mathbf{K}^d$ is finite dimensional if and only if the corresponding space of harmonic functions is finite dimensional.
\end{thm}

Using Theorem 3.1 in \cite{BCG01}, as discussed before the statement of the theorem, we obtain the following corollary with a set of  clear curvature-volume conditions, which is not in the most general form, but which is intended to elucidate the theorem.

\begin{cor}
The conclusions of the theorem are true if the manifold $(M, g)$ has Ricci curvature bounded from below by a constant $-K$ and the volume condition holds:
\[
C r^{D_2} \ge |B(x, r)| \ge c r^{D_1}, \forall r \ge 1, \, x \in M,
\]for some constants $D_2>D_1>0, c, C>0$.
\end{cor}

 As mentioned one can replace the curvature assumption by local volume doubling and local Poincar\'e inequality.

 We also mention that in \cite{Mos21}, it is proven by Mosconi that positive ancient solutions of the heat equation of sub-exponential growth in time must be harmonic if the Ricci curvature is bounded from below by a negative constant.  No volume doubling condition is needed there. But positivity is needed.

  \rmk. 1. A comment for the volume condition in \eqref{eqn-BA} is in order. When $r_2=2 r_1$, the exponential term there becomes a constant and hence it becomes
\begin{equation}
\label{rdouble}
|B(x, 2r)| \le C(1+r)^\alpha |B(x, r)|.
\end{equation} By iteration,  see \eqref{eqn-L5b} below, this implies the growth bound for large geodesic balls:
\[
|B(x, r)| \le C_1 r^{C_2 \alpha \ln(e + r) } |B(x, 1)| \le C_1 e^{C_3 \alpha (\ln(e+r))^2} |B(x, 1)|
\] for some positive constants $C_1, C_2$ and $C_3$.

 2. On the other hand, inequality \eqref{rdouble} alone does not seem to imply
 the exponential term in the volume condition (\ref{eqn-BA}) 1 for all $r>0$. But the exponential term is so large that it hardly is any restriction.

 3. One can also replace condition (\ref{eqn-BA}) 2 (the heat kernel bound) by a modified mean value inequality such as Theorem \ref{thmmmvi} in Section 2, which is broader than the usual one by any polynomial weight.  In \cite{LS84} by Li and Schoen Theorem 1.2,
 it was proven that a mean value inequality with an exponential weight of the radius $e^{C \sqrt{K} r}$  holds for harmonic functions when the Ricci curvature is bounded from below by a negative constant $-K$.
 It is an interesting question to see if the conclusion of the theorem still holds in this case, i.e. if ancient solutions of polynomial growth are polynomials of time when the Ricci curvature is bounded below by a negative constant.

 4.  One naturally wonders if the dimension of $\mathbf{K}^d$ in the theorem is finite. In Section 4, we present an explicit example showing the answer is no in general. As an application,  we revisit the question of the finite dimensionality of the space of harmonic functions of polynomial growth. Using the modified mean value inequality, the slow heat kernel bounds in \cite{BCG01} and \cite{Cho84}, we will also show that the finite dimensionality of the space of harmonic functions in \cite{CM97} and \cite{Li97} are essentially sharp in the sense the volume doubling condition and the mean value inequality can not be relaxed by more than a $\log^2$ factor. This also gives an obstruction to the
 proposal of relaxing the volume doubling condition in \cite{CM98} p118. On the other hand, a positive result on  finite dimensionality is also given with respect to the proposal, c.f. Theorem \ref{thm-ys},  removing all local topological or geometric conditions with respect to a reference point, covering in particular connected sums of $\mathbb{R}^n$ with cylinders.

Theorem \ref{thm-2} will be proven in Section 3 after establishing a modified mean value property in Section 2.  The method of proof for part (a) is based on time derivative estimates of the heat kernel and global integral estimates, which differs from \cite{LZ19}.

 Here is a list of common notations and symbols in the paper. $(M, g)$ is a complete non-compact manifold with metric $g$.
We will use $u=u(x, t)$ to denote a smooth solution to the heat equation $\Delta u - \partial_t u=0$ where $\Delta$ is the Laplace-Beltrami operator; $\nabla f$ denotes the gradient of a function $f$; $d(x, y)$ denotes the distance between two points $x$ and $y$; $B(x, r)$ is the geodesic ball with radius $r$ centered at $x$; $|B(x, r)|$ is the volume of $B(x, r)$; $Q_r(x, t)=
\{ (y, s) \, | \, d(x, y)<r, \, t-r^2 <s <t\}$ is the standard parabolic cube of scale $r$, centered at $(x, t)$; $|Q_r(x, t)|= |B(x, r)| r^2$; $\left<\cdot, \cdot\right>$ denotes the inner product under $g$ for two vector fields; $C, c$ with or without index denote positive constants which may change from line to line.

\vskip24pt

\section{Modified mean value inequality}

\vskip12pt

 In this section we show that a  mean value inequality modified by a polynomial weight of the scale of the parabolic cube continues to hold for manifolds satisfying the assumption of the main result.

\begin{thm}
\label{thmmmvi}
Let $M$ be a complete $n$-dimensional Riemannian manifold satisfying Condition \eqref{eqn-BA}. Fixing $(x, t) \in M \times (0, \infty)$,
there exists a constant $C$,  depending only on the parameters  $\a, \e,\d, m$  in \eqref{eqn-BA}, such that the following statement holds. For any solution $u$ to the heat equation
\begin{equation}\label{eqn-H1}
\D u - \p_s u = 0,
\end{equation}  in $Q_{2r}(x,t)$, $r>0$,
we have
 \begin{equation}\label{eqn-H18}\aligned
|u(x,  t)| \leq    C  (1 +  r^{2m+2\a} ) \left( \frac{1}{|Q_r(x,t)|}\int_{Q_{2r}(x,t)} u^2dyds \right)^{\frac{1}{2}}.
\endaligned
\end{equation}

\end{thm}

\begin{proof}

 Let $\varphi = \varphi(x, t)$ be a standard cut-off function supported in $Q_{\frac{3}{2}r}(x, t)$ such that $\varphi=1$ in $Q_r(x, t)$,  and $|\n \varphi| \leq \frac{C}{r}$,  $| \p_t \varphi | \leq \frac{C}{r^2} $,   here $C$ is a positive constant.   Then we have
\begin{equation*}
(\D-\p_s)(u\varphi) = u (\D-\p_s)\varphi + 2\la \n u, \n \varphi \ra.
\end{equation*}
Let $G$ be the heat kernel on $M$. Then
\begin{equation}\label{eqn-H2}\aligned
u\varphi = & - \int G\left( \D\varphi \cdot u -\p_s \varphi \cdot u + 2\la \n u, \n \varphi \ra \right) dyds \\
=& \int (\la \n G, \n \varphi \ra u - G \la \n u,\n \varphi \ra ) dyds+ \int G\cdot \p_s \varphi \cdot u dyds.
\endaligned
\end{equation}
Hence
\begin{equation}\label{eqn-H3}\aligned
|u(x, t)| \leq & \frac{C}{r} \int_{{\rm Supp}|\n \varphi|}|\n_y G(x, t, y, s)| \cdot |u(y, s)|dyds \\
&+ \frac{C}{r} \int_{{\rm Supp}|\n \varphi|} G(x, t, y, s) \cdot |\n u(y, s)|dyds \\
&+ \frac{C}{r^2} \int_{{\rm Supp}\p_t \varphi} G(x, t, y, s) \cdot | u(y, s)|dyds \\
=:&  I_1 +I_2 +I_3.
\endaligned
\end{equation}  We will bound $I_3, I_2 $ and $I_1$ respectively.

By the condition (\ref{eqn-BA}), we get
\begin{equation}\label{eqn-H4}\aligned
I_3 \leq \frac{C}{r^2}  \int_{{\rm Supp}\p_t \varphi}\frac{(1+t-s)^m e^{-\frac{d^2(x, y)}{(4+\d)(t-s)}}|u(y, s)|}{\sqrt{|B(x, \sqrt{t-s})|}\cdot \sqrt{|B(y, \sqrt{t-s})|}}dyds.
\endaligned
\end{equation}
 Notice that
\[
{\rm Supp}\p_t \varphi \subset B(x, \frac{3}{2}r) \times [t-(\frac{3}{2}r)^2, t-r^2].
\]
Then we have in ${\rm Supp}\p_t \varphi$,

\begin{equation*}
r^2 \leq t-s \leq (\frac{3}{2}r)^2,  \quad  d(x,  y) \leq  \frac{3}{2}r.
\end{equation*}
 It follows that
\[
\left( \sqrt{t-s}+d(x, y) +1 \right)^\a \leq (1+3r)^\a \leq 2^{\a} (1+3^\a r^\a) \leq  6^\a (1+r^\a).
\]
Thus from (\ref{eqn-BA}),  we have  for any $\eta>0$ that
\begin{equation}\label{eqn-H5}\aligned
& |B(x, \sqrt{t-s})| \leq |B(y, \sqrt{t-s} + d(x, y))| \\
\leq & C |B(y, \sqrt{ t-s})| e^{\frac{1}{4+\e}\cdot \frac{(\sqrt{t-s}+d(x, y))^2}{t-s}}\cdot \left( \sqrt{t-s}+d(x, y) +1 \right)^\a \\
\leq &  C  |B(y, \sqrt{ t-s})| e^{\frac{1+\eta}{4+\e} \cdot \frac{d^2(x, y)}{t-s}}\cdot e^{\frac{1+\eta^{-1}}{4+\e}} \left( \sqrt{t-s}+d(x, y) +1 \right)^\a \\
\leq &   C  |B(y, \sqrt{ t-s})| e^{\frac{1+\eta}{4+\e}\cdot \frac{d^2(x, y)}{t-s}}\cdot e^{\frac{1+\eta^{-1}}{4+\e}}  \cdot  6^\a (1+r^\a).
\endaligned
\end{equation}
This implies that
\begin{equation}\label{eqn-H6}\aligned
 \frac{1}{ \sqrt{|B(y, \sqrt{ t-s})|} }
\leq   \frac{C (\sqrt 6)^\a (1+r^{\frac{\a}{2}})  e^{\frac{1+\eta^{-1}}{2(4+\e)}}}{ \sqrt{ |B(x, \sqrt{ t-s})|} }  e^{{\frac{1+\eta}{2(4+\e)}\cdot \frac{d^2(x, y)}{t-s}}}.
\endaligned
\end{equation}
 On the other hand,  since $t-s \geq r^2$ in ${\rm Supp}\p_t \varphi$, we deduce
 \begin{equation}\label{eqn-H7}\aligned
 |B(x, \sqrt{ t-s})|  \geq  |B(x,  r)|.
\endaligned
\end{equation}
By the inequalities (\ref{eqn-H4}),   (\ref{eqn-H6}) and  (\ref{eqn-H7}),   we obtain  
 \begin{equation*}\aligned
I_3 \leq & \frac{C (\sqrt 6)^\a (1+r^{\frac{\a}{2}}) 2^{m} (\frac{3}{2})^{2m}(1+r^{2m})  e^{\frac{1+\eta^{-1}}{2(4+\e)}}}{r^2 |B(x, r)|}\int_{{\rm Supp} \p_t \varphi} e^{-\left( \frac{1}{4+\d} - \frac{1+\eta}{2(4+\e)} \right)\frac{d^2(x, y)}{t-s}}|u(y, s)|dyds \\
= & \frac{C  (\sqrt 6)^\a  2^{m+1} (\frac{3}{2})^{2m}(1+r^{2m+\frac{\a}{2}})  e^{\frac{1+\eta^{-1}}{2(4+\e)}}}{r^2 |B(x, r)|}\int_{{\rm Supp} \p_t \varphi} e^{-\frac{4+2\e-\d-(4+\d)\eta}{2(4+\d)(4+\e)}\frac{d^2(x, y)}{t-s}}|u(y, s)|dyds.
\endaligned
\end{equation*}
Therefore, if $4+2\e-\d>0$, i.e.,
\begin{equation}\label{e-d-e-1}
\e>\frac{\d}{2}-2,
\end{equation}
then we can choose $\eta>0$ depending on $\e$ and $\d$ such that
 \begin{equation}\label{eqn-H8}\aligned
I_3 \leq  &  \frac{  C(\a,\e, \d, m)  (1+ r^{2m+\frac{\a}{2} })}{|Q_r(x, t)|} \int_{Q_{2r}(x, t)} |u(y, s)|dyds.
\endaligned
\end{equation}
Notice that
\begin{equation}\label{eqn-HQ1}\aligned
|Q_{2r}(x, t)| = & (2r)^2 |B(x, 2r)| \\
\leq & (2r)^2\cdot C(1+2r)^\a e^{\frac{4}{4+\e}}|B(x, r)| \\
\leq & C  C(\a, \e)  r^2 (1+r^{\a} )|B(x, r)|.
\endaligned
\end{equation}
The inequalities (\ref{eqn-H8}), (\ref{eqn-HQ1}) and the Cauchy--Schwarz inequality imply that
 \begin{equation}\label{eqn-H8-1}\aligned
I_3 \leq &  \frac{  C(\a,\e, \d, m)  (1+ r^{2m+\frac{\a}{2}} )}{|Q_r(x, t)|} \int_{Q_{2r}(x, t)} |u(y, s)|dyds \\
\leq &  \frac{  C(\a,\e, \d, m)  (1+ r^{2m+\frac{\a}{2}} )}{|Q_r(x, t)|} \left( \int_{Q_{2r}(x, t)} u^2 dyds \right)^{\frac{1}{2}} \cdot \left(  \int_{Q_{2r}(x, t)}  dyds \right)^{\frac{1}{2}} \\
=&  \frac{  C(\a,\e, \d, m)  (1+ r^{2m+\frac{\a}{2}} )}{r^2 |B(x, r)|} \sqrt{|Q_{2r}(x, t)|}\cdot \left( \int_{Q_{2r}(x, t)} u^2 dyds \right)^{\frac{1}{2}} \\
\leq &   C(\a,\e, \d, m)  (1+ r^{2m+\a} ) \left( \frac{1}{|Q_r (x, t)|} \int_{Q_{2r}(x, t)} u^2 dyds \right)^{\frac{1}{2}}.
\endaligned
\end{equation}

 Next we bound $I_2$.
By the Cauchy--Schwarz inequality and Caccioppoli inequality,  we derive
 \begin{equation}\label{eqn-H9}\aligned
I_2 \leq & \frac{C}{r} \left( \int_{{\rm Supp}\n \varphi} G^2 dyds \right)^{\frac{1}{2}} \left( \int_{Q_{\frac{3}{2}r}(x, t)} |\n u|^2 dyds \right)^{\frac{1}{2}} \\
\leq & \frac{C}{r^2} \left( \int_{{\rm Supp}\n \varphi} G^2 dyds \right)^{\frac{1}{2}} \left( \int_{Q_{2r}(x, t)} u^2 dyds \right)^{\frac{1}{2}}.
\endaligned
\end{equation}
By the inequality  (\ref{eqn-H5}), we have
 \begin{equation}\label{eqn-H10}\aligned
&\int_{{\rm Supp}\n \varphi} G^2 dyds \\
\leq & C\int_{{\rm Supp}\n \varphi} \frac{(1+t-s)^{2m}e^{-\frac{2d^2(x, y)}{(4+\d)(t-s)}}}{|B(x, \sqrt{t-s})|\cdot |B(y, \sqrt{t-s})|}dyds \\
\leq & C(\e, \eta) \int_{{\rm Supp}\n \varphi} \frac{(1+t-s)^{2m}}{|B(x, \sqrt{t-s})|^2} e^{-\left( \frac{2}{4+\d} - \frac{1+\eta}{4+\e} \right)\cdot \frac{d^2(x, y)}{t-s}}\cdot (1+d(x, y) +\sqrt{t-s})^\a dyds.
\endaligned
\end{equation}
 Notice that
\[
{\rm supp} \n \varphi \subset Q_{\frac{3}{2}r}(x, t) \backslash B(x,r)\times [t-(\frac{3}{2}r)^2,  t].
\]
Then we have in ${\rm supp}\n \varphi$,
\[
t-s\leq (\frac{3}{2}r)^2, \quad r\leq d(x,y) <\frac{3}{2}r.
\]
Thus from the condition (\ref{eqn-BA}),
\begin{equation*}\label{eqn-H001}\aligned
\frac{\left| B(x, \frac{3}{2}r) \right|}{\left| B(x, \sqrt{t-s}) \right|} \leq & C\left( 1+\frac{3}{2}r \right)^\a e^{\frac{1}{4+\e}\cdot \frac{\left( \frac{3}{2}r \right)^2}{t-s}} \\
\leq & C\left( 1+\frac{3}{2}r \right)^\a e^{\frac{1}{4+\e}\cdot \frac{9}{4}\frac{d^2(x,y)}{t-s}}.
\endaligned
\end{equation*}
It follows that
\begin{equation}\label{eqn-H002}\aligned
\frac{\left| B(x, r) \right|}{\left| B(x, \sqrt{t-s}) \right|} \leq & \frac{\left| B(x, \frac{3}{2}r) \right|}{\left| B(x, \sqrt{t-s}) \right|} \leq   C\left( 1+\frac{3}{2}r \right)^\a e^{\frac{9}{4(4+\e)}\cdot \frac{d^2(x,y)}{t-s}}.
\endaligned
\end{equation}
 Then from (\ref{eqn-H10}), we deduce
 \begin{equation}\label{eqn-H003}\aligned
& \int_{{\rm Supp}\n \varphi} G^2 dyds \\
\leq & \frac{ C(\a,\e, \eta, m) (1+ r^{4m} )(1+\frac{3}{2}r)^{2\a}(1+ r^{\a} )}{\left| B(x, r) \right|^2} \int_{{\rm supp}\n \varphi} e^{-\left( \frac{2}{4+\d} -\frac{1+\eta}{4+\e} -\frac{9}{2(4+\e)} \right)\frac{d^2(x, y)}{t-s}}dyds \\
\leq & \frac{ C(\a,\e, \eta, m) \left(1+ r^{4m+ 3\a}  \right)}{\left| B(x, r) \right|^2} \int_{Q_{\frac{3}{2}r}(x, t)} e^{-\left( \frac{2}{4+\d}  -\frac{11+2\eta}{2(4+\e)} \right)\frac{d^2(x, y)}{t-s}}dyds.
\endaligned
\end{equation}
If  $\e > \frac{11}{4}\d +7$, then we can choose $\eta>0$  depending only on $\e$ and $\d$ such that $ \frac{2}{4+\d}  -\frac{11+2\eta}{2(4+\e)}>0$. Then from (\ref{eqn-H003}),
 \begin{equation}\label{eqn-H11}\aligned
 \int_{{\rm Supp}\n \varphi} G^2 dyds
\leq   \frac{ C(\a,\e, \d, m)  \left(1+ r^{4m+3\a}  \right)}{\left| B(x, r) \right|^2} \cdot \left| Q_{\frac{3}{2}r}(x, t)\right|.
\endaligned
\end{equation}

By the condition (\ref{eqn-BA}),
\begin{equation*}
 \left|B(x, \frac{3}{2}r)\right| \leq C \left( \frac{3}{2}r+1 \right)^\a \cdot e^{\frac{1}{4+\e}\cdot \frac{9}{4}}|B(x, r)| \leq C  C(\a,\e)  (1+ r^{\a} )|B(x, r)|.
\end{equation*}
Therefore
\begin{equation}\label{eqn-H12}
\left|Q_{\frac{3}{2}r}(x,t)\right| = \left( \frac{3}{2}r \right)^2 \left|B(x, \frac{3}{2}r)\right| \leq C   C(\a,\e)  (1+  r^\a )r^2|B(x, r)|.
\end{equation}
Substituting (\ref{eqn-H12}) into (\ref{eqn-H11}),  we obtain
\begin{equation}\label{eqn-H13}\aligned
\int_{{\rm Supp}\n \varphi} G^2 dyds
\leq \frac{  C(\a,\e, \d, m)  \left( 1+  r^{4m+4\a}  \right) r^2}{|B(x, r)|}.
\endaligned
\end{equation}
Combining (\ref{eqn-H9}) with (\ref{eqn-H13}), it follows
 \begin{equation}\label{eqn-H14}\aligned
I_2 \leq & \frac{  C(\a,\e, \d, m)  (1+ r^{2m+2\a} )}{r \sqrt{|B(x, r)|}} \cdot \left( \int_{Q_{2r}(x,t)} u^2dyds \right)^{\frac{1}{2}} \\
=&   C(\a,\e, \d, m)  (1+ r^{2m+2\a} )\left( \frac{1}{|Q_r(x,t)|}\int_{Q_{2r}(x,t)} u^2 dyds \right)^{\frac{1}{2}}.
\endaligned
\end{equation}

 Finally we bound $I_1$.
By the Caccioppoli inequality,
 \begin{equation}\label{eqn-H15}\aligned
&\int_{{\rm Supp} \n \varphi } | \n_y G(x, t, y, s) |^2 dyds \\
\leq &  \int_{(Q_{\frac{3}{2}r}(x,t)\backslash Q_r(x,t)) \cap {\rm Supp} \n \varphi} | \n_y G(x, t, y, s) |^2 dyds  \\
\leq & \frac{C}{r^2} \int_{(Q_{2r}(x, t)\backslash Q_{\frac{1}{2}r}(x, t)) \cap \{d(x, y)>r/4\}} G^2(x, t, y, s)dyds \\
\leq & \frac{C}{r^2} \int_{(Q_{2r}(x, t)\backslash Q_{\frac{1}{2} r}(x, t)) \cap \{d(x, y)>r/4\}} \frac{(1+t-s)^{2m}e^{-\frac{2d^2(x, y)}{(4+\d)(t-s)}}}{|B(x, \sqrt{t-s})|\cdot |B(y, \sqrt{t-s})|}dyds.
\endaligned
\end{equation}
 Then following the same argument above as we obtain (\ref{eqn-H13}), we have from (\ref{eqn-H5}),  (\ref{eqn-HQ1}), (\ref{eqn-H002}) and (\ref{eqn-H15}) that
 \begin{equation}\label{eqn-H16}\aligned
&\int_{{\rm Supp} \n \varphi } | \n_y G(x, t, y, s) |^2 dyds \\
\leq & \frac{  C(\e, \eta)  (1+ r^{4m} )}{r^2} \int_{[Q_{2r}(x, t)\backslash Q_{\frac{1}{2}r}(x, t)] \cap \{d(x, y)>r/4 \}} \frac{1}{|B(x, \sqrt{t-s})|^2} e^{-\left(\frac{2}{4+\d}- \frac{1+\eta}{4+\e}\right)\cdot \frac{d^2(x, y)}{t-s}} \\
& \cdot \left( 1+d(x, y) +\sqrt{t-s} \right)^\a dyds \\
\leq & \frac{  C(\a,\e, \d, m)  (1+ r^{4m+3\a})}{r^2 |B(x, r)|^2} \cdot |Q_{2r}(x, t)| \\
\leq  & \frac{  C(\a,\e, \d, m)  (1+ r^{4m+3\a})}{r^2 |B(x, r)|^2} \cdot C C(\a,\e)  r^2 (1+ r^\a )|B(x, r)| \\
\leq  & \frac{ C(\a,\e, \d, m)  (1+ r^{4m+4\a})}{ |B(x, r)|}.
\endaligned
\end{equation}

The Cauchy-Schwarz inequality and (\ref{eqn-H16}) imply that
 \begin{equation}\label{eqn-H17}\aligned
I_1 \leq & \frac{C}{r} \left(\int_{{\rm Supp}\n \varphi} | \n_y G(x, t, y, s) |^2dyds \right)^{\frac{1}{2}} \left( \int_{Q_{2r}(x, t)} u^2dyds \right)^{\frac{1}{2}} \\
\leq &   C(\a,\e, \d, m)  (1+ r^{2m+2\a} ) \left( \frac{1}{|Q_r(x,t)|}\int_{Q_{2r}(x,t)} u^2dyds \right)^{\frac{1}{2}}.
\endaligned
\end{equation}
Hence by (\ref{eqn-H3}), (\ref{eqn-H8-1}), (\ref{eqn-H14}),  (\ref{eqn-H17})  and the Young inequality,   we have
 \begin{equation*}\aligned
|u(x,  t)| \leq    C(\a,\e, \d, m)  (1 + r^{2m+2\a}) \left( \frac{1}{|Q_r(x,t)|}\int_{Q_{2r}(x,t)} u^2dyds \right)^{\frac{1}{2}}.
\endaligned
\end{equation*}
\end{proof}


\section{ Space of ancient solutions}

\vskip12pt

 In this section we will prove the first main result. \\

\noindent{\bf Proof of Theorem \ref{thm-2} (a).}
The idea of the proof is different from that in \cite{LZ19} where a local integral method is used. Here the proof is based on  time derivatives estimates of the heat kernel $G$ and global integral estimates.

\noindent{\bf Step 1.}  We show the estimates: for all $x, y \in M$, $t>0$, $k=1, 2, 3,...$, there is a constant $C>0$ depending only on the parameters in the theorem, such that
\begin{equation}\label{eqn-HA25a}\aligned
\left| \p_t^{k}G(x, t, y) \right| \leq & \frac{ C C^k (2+\sqrt{2})^{(k+1)(k+2)} \left( 1+t^{2m+ \frac{7\a}{4}} \right)}{t^{k}\sqrt{\left| B(x, \sqrt t) \right|\cdot \left| B(y, \sqrt t) \right|}} \cdot e^{-\frac{d^2(x, y)}{(4+2\d)t}}.
\endaligned
\end{equation}

 Fixing $r \le \sqrt{t}/16$ and another parameter $\sigma \in (1, 2]$ to be chosen later, let $\psi$ be a standard cut-off function supported in $Q_{\sigma r} \equiv Q_{ \sm r}(x, t)$,  such that $\psi=1$ in $Q_{r}(x, t)$ and $|\n \psi| \leq \frac{C}{(\sm-1)r},  |\p_t \psi|\leq \frac{C}{[(\sm-1) r]^2}$.  Then we have
 \begin{equation}\label{eqn-HA1}\aligned
\int_{Q_{\sm r}} (\D u)^2 \psi^2 dz d\tau = & \int_{Q_{\sm r}} u_t \cdot \D u \cdot \psi^2 dz d\tau \\
=& - \int_{Q_{\sm r}} \la \n u, \n(u_t \psi^2) \ra dz d\tau \\
=& - \int_{Q_{\sm r}} \la \n u, \n u_t \ra \psi^2dz d\tau - 2 \int_{Q_{\sm r}} \la \n u, \n \psi \ra \psi u_t dz d\tau \\
=& -\frac{1}{2} \int_{Q_{\sm r}} (|\n u|^2)_t \psi^2 dz d\tau - 2\int_{Q_{\sm r}} u_t \psi \la \n u, \n \psi \ra dz d\tau \\
\leq & \frac{1}{2}\int_{Q_{ \sigma  r}} |\n u|^2 \cdot ( \psi^2)_t dz d\tau + \frac{1}{2}\int_{Q_{\sm r}} u_t^{2} \psi^2 dz d\tau \\
&+2 \int_{Q_{\sm r}} |\n u|^2 |\n \psi|^2 dz d\tau.
\endaligned
\end{equation}
 Here and later we also used $(z, \tau)$ as  integration variables for simplicity of notations unless confusion arises.
 It follows that
 \begin{equation}\label{eqn-HA2}\aligned
\int_{Q_{\sm r}} (\D u)^2 \psi^2 dz d\tau = &\int_{Q_{\sm r}} u_t^{2} \psi^2 dz d\tau \\
\leq & \int_{Q_{\sm r}} |\n u|^2 \cdot ( \psi^2)_t dz d\tau + 4 \int_{Q_{\sm r}} |\n u|^2 |\n \psi|^2 dz d\tau \\
\leq & \frac{C}{[(\sm-1) r]^2}\int_{Q_{\sm r}}|\n u|^2 dz d\tau.
\endaligned
\end{equation}
 By the Caccioppoli inequality,
 \begin{equation}\label{eqn-HA3}\aligned
\int_{Q_{\sm r}} |\n u|^2 dz d\tau \leq \frac{C}{[(\sm-1) r]^2}\int_{Q_{\sm^2 r}} u^2 dz d\tau.
\endaligned
\end{equation}
 From (\ref{eqn-HA2}) and (\ref{eqn-HA3}),  we get
 \begin{equation}\label{eqn-HA4}\aligned
\int_{Q_{r}} (\p_t u)^2 dz d\tau = \int_{Q_{r}} (\D u)^2 dz d\tau \leq \frac{C}{[(\sm-1)r]^4} \int_{Q_{\sm^2 r}}u^2 dz d\tau.
\endaligned
\end{equation}

From  inequality (\ref{eqn-HA4}), we choose
\[
\sm_j=\left( 1+\sum^j_{i=0}  (2+\sqrt 2)^{-i-1} \right)^2,
\] $j=0, 1, ..., k$, then
\begin{equation}\label{eqn-HA18}
\aligned
\int_{Q_r}(\p_t^{2} u)^2 dz d\tau& = \int_{Q_r} (\D u_t)^2 dz d\tau \leq \frac{C}{[(\sqrt{\sm_1}-1)r]^4}\int_{Q_{\sm_1 r}} u_t^{2}dz d\tau \\
 &\leq \frac{C^2}{[(\sqrt{\sm_1}-1)r]^4 [(\sqrt{\sm_2}-\sqrt{\sm_1})r]^4} \int_{Q_{\sm_2 \, r}} u^2 dz d\tau.
\endaligned
\end{equation}
By induction, we deduce
\begin{equation}\label{eqn-HA19}\aligned
\int_{Q_r}(\p_t^{k} u)^2 dz d\tau  \leq \frac{C^k (2+\sqrt{2})^{2(k+1)(k+2)}}{r^{4k}} \int_{Q_{\sm_k \, r}} u^2 dz d\tau.
\endaligned
\end{equation}
Since $\p_t^{k}u$ is also a solution to the heat equation, then by the modified mean value inequality and (\ref{eqn-HA19}),
\begin{equation}\label{eqn-HA20a}\aligned
\left| \p_t^{k}u (x, t)\right|^2 \leq & \frac{C(1+t^{2m+ 2\a})}{|Q_r(x, t)|} \int_{Q_{2r}}\left| \p_t^{k}u \right|^2 dz d\tau \\
\leq & \frac{C \tilde{C}_k (1+t^{2m+ 2\a})}{|Q_r(x, t)|}\cdot \frac{1}{r^{4k}} \int_{Q_{ 4r}} u^2 dz d\tau,
\endaligned
\end{equation}where the fact $\sm_k r < 2 r$ is used and $\tilde{C}_k= C^k (2+\sqrt{2})^{2(k+1)(k+2)}$.
Applying $G$ in the above inequality (\ref{eqn-HA20a}),
\begin{equation}\label{eqn-HA20}\aligned
\left| \p_t^{k}G(x, t, y) \right|^2 \leq & \frac{C \tilde{C}_k (1+t^{2m+ 2\a})}{|Q_r(x, t)|}\cdot \frac{1}{r^{4k}} \int_{Q_{4r}(x, t)} G^2(z,\tau,y) dzd\tau \\
\leq &  \frac{C \tilde{C}_k (1+t^{2m+ 2\a})}{|Q_r(x, t)|}\cdot \frac{1}{r^{4k}} \int_{Q_{4r}(x,t)} \frac{(1+\tau)^{2m}e^{-\frac{2d^2(z,y)}{(4+\d)\tau}}}{|B(z,\sqrt{\tau})|\cdot |B(y, \sqrt{\tau})|}dzd\tau \\
\leq & \frac{C \tilde{C}_k (1+t^{4m+ 2\a})}{|Q_r(x, t)|}\cdot \frac{1}{r^{4k}} \int_{Q_{4r}(x,t)} \frac{e^{-\frac{2d^2(z,y)}{(4+\d)\tau}}}{|B(z,\sqrt{\tau})|\cdot |B(y, \sqrt{\tau})|}dzd\tau.
\endaligned
\end{equation}
Let $(z, \tau) \in Q_{2r}(x, t)$,  $r=\frac{\sqrt t}{16}$.
By the triangle inequality and our volume condition, we obtain
\begin{equation}\label{eqn-HA21}\aligned
\left| B(x, \sqrt t) \right| \leq & \left| B(z, \frac{3}{2}\sqrt t) \right| \leq C\left( 1+t^\frac{\a}{2} \right) \left| B(z, \sqrt \tau) \right|, \\
\left| B(y, \sqrt t) \right| \leq &  C\left( 1+t^\frac{\a}{2} \right) \left| B(y, \sqrt \tau) \right|
\endaligned
\end{equation}
and
\begin{equation*}\aligned
d^2(z,y) \geq & (1-\widetilde\d_0)d^2(x,y) -C\tau,
\endaligned
\end{equation*}
where  $\widetilde\d_0 < \frac{\d}{4+2\d}$ is another positive constant.  Then we derive
\begin{equation}\label{eqn-HA22}\aligned
e^{-\frac{2d^2(z,y)}{(4+\d)\tau}} \leq e^{-\frac{2\left( (1-\widetilde\d_0)d^2(x, y)-C\tau \right)}{(4+\d)\tau}} = e^{\frac{2C}{4+\d}}\cdot e^{-\frac{2d^2(x,y)}{\frac{4+\d}{1-\widetilde\d_0}\tau}}< C e^{-\frac{2d^2(x, y)}{(4+2\d)t}}.
\endaligned
\end{equation}
From (\ref{eqn-HA20})--(\ref{eqn-HA22}),  we have
\begin{equation}\label{eqn-HA23}\aligned
\left| \p_t^{k}G(x, t, y) \right|^2 \leq & \frac{C \tilde{C}_k (1+t^{4m+ 3\a})}{r^{4k}\cdot r^2|B(x, r)|} \cdot \frac{e^{-\frac{2d^2(x, y)}{(4+2\d)t}}}{\left| B(x, \sqrt t) \right|\cdot \left| B(y, \sqrt t) \right|} \cdot \left| Q_{4r}(x, t) \right|.
\endaligned
\end{equation}
By the condition (\ref{eqn-BA}),
\begin{equation}\label{eqn-HA24}\aligned
\frac{ \left| Q_{ 4r}(x, t) \right|}{r^{2}|B(x, r)|} = & \frac{16 \left| B(x, 4r) \right|}{|B(x, r)|}
\leq  C (1+ 4r)^\a e^{\frac{16}{4+\e}}
\leq  C (1+ t^{\frac{\a}{2}}).
\endaligned
\end{equation}
Substituting (\ref{eqn-HA24}) into (\ref{eqn-HA23}),
\begin{equation*}\aligned
\left| \p_t^{k}G(x, t, y) \right|^2 \leq & \frac{ C C^k (2+\sqrt 2)^{2(k+1)(k+2) } \left(1+t^{4m+ \frac{7\a}{2}}\right)}{t^{2k}\left| B(x, \sqrt t) \right|\cdot \left| B(y, \sqrt t) \right|} \cdot e^{-\frac{2d^2(x, y)}{(4+2\d)t}}.
\endaligned
\end{equation*}
Namely,
\begin{equation*}\label{eqn-HA25}\aligned
\left| \p_t^{k}G(x, t, y) \right| \leq & \frac{ C C^k (2+\sqrt 2)^{(k+1)(k+2)} \left( 1+t^{2m+ \frac{7\a}{4}} \right)}{t^{k}\sqrt{\left| B(x, \sqrt t) \right|\cdot \left| B(y, \sqrt t) \right|}} \cdot e^{-\frac{d^2(x, y)}{(4+2\d)t}}.
\endaligned
\end{equation*}

\noindent{\bf Step 2.}

 Note the assumption on the growth of $u$ is assumed only for $t \le 0$. We shall deduce that for all $t>0$ we have similar bounds for the extended solution
\begin{equation}\label{eqn-HAC1}
|u(x, t)| \leq C(\d,\e, d, k, m, \a) \left( 1+t^{\frac{m}{2}+\frac{3\a}{8}} \right) \left( t^{\frac{d}{2}} + d^{d}(x, 0) + 1 \right).
\end{equation}

 Since the manifold $M$ be may not be volume doubling, we need to explain why the solution can be extended uniquely.

 When $r_2=2 r_1$, the exponential term in our condition (\ref{eqn-BA})   becomes a constant and hence
\[
|B(x, 2r)| \le C(1+r)^\alpha |B(x, r)|.
\]By iteration, as done in \eqref{eqn-L5b} below, this implies the growth bound for $r \ge 1$:
\[
|B(0, r)| \le C_1 r^{C_2 \alpha [\ln(e + r)] } |B(0, 1)|
\] for some positive constants $C_1, C_2$. Here $0$ is a reference point in $M$. So the volume of the balls have sub-exponential growth. Using this and our growth condition for $u$ \eqref{eqn-HAC1}, we can conclude from Theorem 9.2 in \cite{Gri} that $u$ can be uniquely extended to all $t>0$.  By the same theorem we also have the stochastic completeness, i.e. $\int_M G(x, t, y) dy =1$.

 Fix $(x, t) \in M\times (-\infty,  \infty)$.  Now we choose $r=\sqrt{|t|}$.  Pick a large positive number $T> 8^k r^2$.  Note that for $t>-T$,
  \begin{equation}\label{eqn-HA27}\aligned
u(x, t) = \int G(x, t, y, -T)u(y,-T)dy = \int G(x, t+T, y) u(y, -T)dy.
\endaligned
\end{equation}
Then we have
  \begin{equation}\label{eqn-HA28}\aligned
u^2(x, t)\leq & \left( \int Gdy \right) \left( \int G(x, t+T, y) u^2(y, -T)dy \right) \\
=& \int G(x, t+T, y) u^2(y, -T)dy  \\
\leq & C\int G(x, t+T, y) \left( d^2(y, 0)+ T +1  \right)^d dy \\
\leq & C \int \frac{(1+t+T)^m e^{-\frac{d^2(x, y)}{(4+\d)(t+T)}}}{\sqrt{\left| B(x, \sqrt{t+T}) \right|\cdot \left| B(y,\sqrt{t+T}) \right|}}\left( d^2(y,x)+d^2(x,0) +T+1 \right)^d dy.
\endaligned
\end{equation}
It follows that
  \begin{equation}\label{eqn-HA29}\aligned
u^2(x, t)\leq & \frac{C\cdot 2^d (1+t+T)^m}{\sqrt{\left| B(x, \sqrt{t+T}) \right|}} \int \frac{e^{-\frac{d^2(x,y)}{(4+\d)(t+T)}}}{\sqrt{\left| B(y, \sqrt{t+T}) \right|}} d^{2d}(x,y)dy \\
& + \frac{C\cdot 2^d \left( d^{2d}(x, 0)+ (T+1)^d \right)(1+t+T)^m}{\sqrt{\left| B(x, \sqrt{t+T}) \right|}} \int \frac{e^{-\frac{d^2(x,y)}{(4+\d)(t+T)}}}{\sqrt{\left| B(y, \sqrt{t+T}) \right|}} dy.
\endaligned
\end{equation}
Let
  \begin{equation*}\aligned
I_1 :=  \frac{1}{\sqrt{\left| B(x, \sqrt{t+T}) \right|}} \int \frac{e^{-\frac{d^2(x,y)}{(4+\d)(t+T)}}}{\sqrt{\left| B(y, \sqrt{t+T}) \right|}} d^{2d}(x,y)dy.
\endaligned
\end{equation*}
By a simple analysis, we deduce
  \begin{equation}\label{eqn-HA30}\aligned
I_1 \leq \frac{C_{\d}^d(t+T)^d\cdot C}{\sqrt{\left| B(x, \sqrt{t+T}) \right|}} \int \frac{e^{-\frac{d^2(x,y)}{(4+2\d)(t+T)}}}{\sqrt{\left| B(y, \sqrt{t+T}) \right|}} dy,
\endaligned
\end{equation}
where $C_\d := \frac{1}{\frac{1}{4+\d}- \frac{1}{4+2\d}} = \frac{1}{\d}(4+\d)(4+2\d) > 0$.  

For any $z \in B(x, \sqrt{t+T})$,  $d(x, z) \leq \sqrt{t+T}$. Then we get
\[
d(y, z) \leq d(x,z) + d(x,y) \leq d(x,y) + \sqrt{t+T}.
\]
That is  $z\in B(y, \sqrt{t+T}+d(x,y))$.  Thus $ B(x, \sqrt{t+T}) \subset B(y, \sqrt{t+T}+d(x,y))$.  Hence by the  condition  (\ref{eqn-BA}),
  \begin{equation}\label{eqn-HA31}\aligned
\left| B(x, \sqrt{t+T}) \right| \leq & \left| B(y, \sqrt{t+T}+d(x,y)) \right|  \\
\leq & C\left( 1+\sqrt{t+T} +d(x, y) \right)^\a \cdot e^{\frac{1}{4+\e}\cdot \frac{\left( \sqrt{t+T}+d(x,y)\right)^2}{t+T}} \cdot \left| B(y, \sqrt{t+T}) \right|.
\endaligned
\end{equation}
It follows that for any $\eta>0$
  \begin{equation}\label{eqn-HA32}\aligned
\frac{1}{\sqrt{\left| B(y, \sqrt{t+T}) \right|}} \leq &  C\left( 1+\sqrt{t+T} +d(x, y) \right)^{\frac{\a}{2}}e^{\frac{1}{2(4+\e)}\cdot \frac{\left( \sqrt{t+T}+d(x,y)\right)^2}{t+T}} \cdot \frac{1}{\sqrt{\left| B(x, \sqrt{t+T}) \right|}} \\
\leq& C\left( 1+\sqrt{t+T} +d(x, y) \right)^{\frac{\a}{2}}e^{\frac{1+\eta}{2(4+\e)}\cdot \frac{d^2(x, y)}{t+T}}\cdot e^{\frac{1+\eta^{-1}}{2(4+\e)}} \cdot \frac{1}{\sqrt{\left| B(x, \sqrt{t+T}) \right|}}.
\endaligned
\end{equation}
Substituting (\ref{eqn-HA32}) into (\ref{eqn-HA30}),
  \begin{equation}\label{eqn-HA33}\aligned
I_1 \leq & \frac{C_{\d}^d(t+T)^d C\cdot C_{\e,\eta}}{\left| B(x, \sqrt{t+T}) \right|} \int \left( 1+\sqrt{t+T} +d(x, y) \right)^{\frac{\a}{2}} \cdot e^{-\left( \frac{1}{4+2\d} -\frac{1+\eta}{2(4+\e)} \right)\frac{d^2(x,y)}{t+T}}dy \\
\leq &  \frac{C_{\d}^d(t+T)^d C\cdot C_{\e,\eta}}{\left| B(x, \sqrt{t+T}) \right|} \left\{ \int Ce^{-\left( \frac{1}{4+2\d} -\frac{1+\eta}{2(4+\e)} \right)\frac{d^2(x,y)}{t+T}}dy \right. \\
& \left.  +C(\sqrt{t+T})^{\frac{\a}{2}} \int \left( 1+ \left( \frac{d(x,y)}{\sqrt{t+T}} \right)^{\frac{\a}{2}} \right)e^{-\left( \frac{1}{4+2\d} -\frac{1+\eta}{2(4+\e)} \right)\frac{d^2(x,y)}{t+T}}dy   \right\}.
\endaligned
\end{equation}
 Suppose that there exists $\gamma>0$ such that $\frac{1}{4+2\d} -\frac{1+\eta}{2(4+\e)} > \gamma$.
 Then by a simple analysis, we conclude
  \begin{equation}\label{eqn-HA34}\aligned
I_1 \leq \frac{C(\d, \e,  \eta, d)(t+T)^d \left( 1+(t+T)^{\frac{\a}{4}} \right)}{\left| B(x, \sqrt{t+T}) \right|} \int_M e^{- \gamma \cdot \frac{d^2(x,y)}{t+T}}dy.
\endaligned
\end{equation}
Since
  \begin{equation}\label{eqn-HA35}\aligned
\int_M e^{- \gamma \cdot \frac{d^2(x,y)}{t+T}}dy = & \sum_{i=1}^\infty \int_{B(x, 2^i \sqrt{t+T})\backslash B(x, 2^{i-1} \sqrt{t+T})} e^{- \gamma \cdot \frac{d^2(x,y)}{t+T}}dy \\
& + \int_{B(x, \sqrt{t+T})} e^{- \gamma \cdot \frac{d^2(x,y)}{t+T}}dy,
\endaligned
\end{equation}
 therefore we obtain
  \begin{equation}\label{eqn-HA36}\aligned
\frac{1}{\left| B(x, \sqrt{t+T}) \right|}\int_M e^{- \gamma \cdot \frac{d^2(x,y)}{t+T}}dy \leq \sum_{i=1}^\infty e^{- \gamma \cdot 2^{2(i-1)}} \cdot \frac{\left| B(x, 2^i\sqrt{t+T}) \right|}{B(x, \sqrt{t+T})} +1.
\endaligned
\end{equation}
From  (\ref{eqn-BA}),  we derive
  \begin{equation}\label{eqn-HA37}\aligned
 \frac{\left| B(x, 2^i\sqrt{t+T}) \right|}{\left| B(x, \sqrt{t+T})\right|} \leq C\left( 1+\left( 2^i \sqrt{t+T} \right)^\a \right) e^{\frac{1}{4+\e}\cdot (2^i)^2}.
\endaligned
\end{equation}
Substituting (\ref{eqn-HA37}) into (\ref{eqn-HA36}), we get
  \begin{equation}\label{eqn-HA38}\aligned
\frac{1}{\left| B(x, \sqrt{t+T}) \right|}\int_M e^{- \gamma \cdot \frac{d^2(x,y)}{t+T}}dy \leq & \sum_{i=1}^\infty C e^{-\left( \frac{ \gamma}{4} - \frac{1}{4+\e} \right) \cdot 2^{2i}} \cdot 2^{\a i} (t+T)^{\frac{\a}{2}} \\
& + \sum_{i=1}^\infty C e^{-\left( \frac{ \gamma}{4} - \frac{1}{4+\e} \right) \cdot 2^{2i}} +1.
\endaligned
\end{equation}
If
\[
\e_0:= \frac{ \gamma}{4} - \frac{1}{4+\e} > 0,
\]
then from (\ref{eqn-HA38}),  we obtain
  \begin{equation}\label{eqn-HA39}\aligned
\frac{1}{\left| B(x, \sqrt{t+T}) \right|}\int_M e^{-\gamma \cdot \frac{d^2(x,y)}{t+T}}dy \leq & \sum_{i=1}^\infty C e^{-\e_0 \cdot 2^{2i}} \cdot 2^{\a i} (t+T)^{\frac{\a}{2}} \\
& + \sum_{i=1}^\infty C e^{- \e_0\cdot 2^{2i}} +1.
\endaligned
\end{equation}
Since $\sum_{i=1}^\infty e^{-\e_0 \cdot 2^{2i}} \cdot 2^{\a i} $ and $\sum_{i=1}^\infty e^{-\e_0 \cdot 2^{2i}}$ are both convergent,  therefore we get
 \begin{equation}\label{eqn-HA40}\aligned
\frac{1}{\left| B(x, \sqrt{t+T}) \right|}\int_M e^{-\gamma \cdot \frac{d^2(x,y)}{t+T}}dy \leq & C_1(t+T)^{\frac{\a}{2}} +C_2.
\endaligned
\end{equation}
 Before proceeding further, we determine the relation between $\e$ and $\d$. The above computation shows that we require that there exist positive constants $\eta$, $\gamma$ and $\e_0$ such that
\begin{equation*}
 \begin{cases}
    \frac{1}{4+2\d} -\frac{1+\eta}{2(4+\e)} > \gamma>0,\\
    \e_0= \frac{ \gamma}{4} - \frac{1}{4+\e} > 0,\\
    \eta>0.
 \end{cases}
\end{equation*}
This can be achieved if
\begin{equation*}
    \frac{1}{4+2\d} -\frac{1}{2(4+\e)} > \frac{4}{4+\e},
\end{equation*}
which is equivalent to that
\begin{equation}\label{e-d-e-2}
    \e>9\d+14.
\end{equation}

Substituting (\ref{eqn-HA40}) into (\ref{eqn-HA34}),
  \begin{equation}\label{eqn-HA41}\aligned
I_1 \leq C(\d, \e, d)(t+T)^d \left( 1+(t+T)^{\frac{3\a}{4}} \right).
\endaligned
\end{equation}
Let
  \begin{equation*}\aligned
I_2 :=  \frac{1}{\sqrt{\left| B(x, \sqrt{t+T}) \right|}} \int \frac{e^{-\frac{d^2(x,y)}{(4+\d)(t+T)}}}{\sqrt{\left| B(y, \sqrt{t+T}) \right|}}dy.
\endaligned
\end{equation*}
From (\ref{eqn-HA32}), we derive
  \begin{equation}\label{eqn-HA42}\aligned
I_2 \leq & \frac{ C\cdot C_{\e, \eta}}{\left| B(x, \sqrt{t+T}) \right|} \int \left( 1+\sqrt{t+T} +d(x, y) \right)^{\frac{\a}{2}} \cdot e^{-\left( \frac{1}{4+\d} - \frac{1+\eta}{2(4+\e)} \right)\frac{d^2(x,y)}{t+T}}dy.
\endaligned
\end{equation}
Then by a similar proof to get (\ref{eqn-HA41}), we can conclude that
  \begin{equation}\label{eqn-HA43}\aligned
I_2 \leq & \frac{ C\cdot C_{\e, \eta} \left( 1+(t+T)^{\frac{\a}{4}} \right)}{\left| B(x, \sqrt{t+T}) \right|} \int_M  e^{-  \gamma \cdot \frac{d^2(x,y)}{t+T}}dy \\
\leq & C(\d, \e)  \left( 1+(t+T)^{\frac{3\a}{4}} \right).
\endaligned
\end{equation}
Combining (\ref{eqn-HA29}), (\ref{eqn-HA41}) with (\ref{eqn-HA43}), we obtain (\ref{eqn-HAC1}).

\bigskip

\noindent{\bf Step 3.} We shall give the estimate of the $k^{\rm th}$ time derivative of $u$.

 We fix $(x_0, t_0) \in M\times (-\infty,  0)$ and choose $T>0$ sufficiently large.  Then
\begin{equation}\label{eqn-HA44}\aligned
u(x_0, t_0) = \int G(x_0, t_0+T,  y) u(y, -T)dy
\endaligned
\end{equation}
and
  \begin{equation}\label{eqn-HA45}\aligned
\p^k_{t_0} u(x_0, t_0) = \int \p^k_{t_0} G(x_0, t_0+T,  y) u(y, -T)dy.
\endaligned
\end{equation}
Applying (\ref{eqn-HA25a}) on the above inequality (\ref{eqn-HA45}) with $t=t_0+T$,
  \begin{equation}\label{eqn-HA46}\aligned
\left| \p^k_{t_0} u(x_0, t_0) \right| \leq &  \frac{C\cdot C_{k,\e}\left( 1+(t_0+T)^{2m+ \frac{7\a}{4}} \right) }{(t_0+T)^k \sqrt{\left| B(x_0, \sqrt{t_0+T}) \right|}} \int \frac{e^{-\frac{d^2(x_0, y)}{(4+2\d)(t_0+T)}} }{ \sqrt{ \left| B(y, \sqrt{t_0+T}) \right|}} \\
& \cdot \left( 1+d^2(y,0)+T \right)^{\frac{d}{2}}dy.
\endaligned
\end{equation}
Consider
  \begin{equation}\label{eqn-HA47}\aligned
I =   \frac{1}{\sqrt{\left| B(x_0, \sqrt{t_0+T}) \right|}} \int \frac{e^{-\frac{d^2(x_0, y)}{(4+2\d)(t_0+T)}} }{ \sqrt{ \left| B(y, \sqrt{t_0+T}) \right|}}  \cdot \left( 1+d^2(y,0) +T \right)^{\frac{d}{2}}dy.
\endaligned
\end{equation}
 Direct computation similar to Step 2 gives us
  \begin{equation}\label{eqn-HA48}\aligned
I \leq &  \frac{ 2^{\frac{d}{2}}}{ \sqrt{\left| B(x_0, \sqrt{t_0+T}) \right|}} \int \frac{e^{-\frac{d^2(x_0, y)}{(4+2\d)(t_0+T)}} }{ \sqrt{ \left| B(y, \sqrt{t_0+T}) \right|}}d^d(x_0, y)dy \\
& + \frac{2^{\frac{d}{2}} \left( d^d(x_0, 0) + (T+1)^{\frac{d}{2}} \right)}{\sqrt{\left| B(x_0, \sqrt{t_0+T}) \right|}} \int  \frac{e^{-\frac{d^2(x_0, y)}{(4+2\d)(t_0+T)}} }{ \sqrt{ \left| B(y, \sqrt{t_0+T}) \right|}}dy.
\endaligned
\end{equation}
Let
  \begin{equation*}\aligned
\widetilde I_1 :=  \frac{1}{\sqrt{\left| B(x_0, \sqrt{t_0+T}) \right|}} \int \frac{e^{-\frac{d^2(x_0,y)}{(4+2\d)(t_0+T)}}}{\sqrt{\left| B(y, \sqrt{t_0+T}) \right|}} d^{d}(x_0,y)dy.
\endaligned
\end{equation*}
 Then following a similar argument as we obtain (\ref{eqn-HA30}),  we conclude
 \begin{equation}\label{eqn-HHA1}\aligned
\widetilde I_1 \leq \frac{\widetilde C_\d^{\frac{d}{2}}(t_0+T)^{\frac{d}{2}}\cdot C}{\sqrt{\left| B(x_0, \sqrt{t_0+T}) \right|}} \int \frac{e^{-\frac{d^2(x_0,y)}{(4+3\d)(t_0+T)}}}{\sqrt{\left| B(y, \sqrt{t_0+T}) \right|}}dy,
\endaligned
\end{equation}
where $\widetilde C_\d := \frac{1}{\frac{1}{4+2\d}- \frac{1}{4+3\d}} = \frac{1}{\d}(4+2\d)(4+3\d) > 0$.

  Next by a similar proof to get (\ref{eqn-HA33}),   we derive  for any $\tilde{\eta}>0$ that
 \begin{equation}\label{eqn-HHA2}\aligned
\widetilde I_1 \leq & \frac{\widetilde C_\d^{\frac{d}{2}}(t_0+T)^{\frac{d}{2}}C\cdot C_{\e,\tilde{\eta}}}{\left| B(x_0, \sqrt{t_0+T}) \right|} \int \left(1+ \sqrt{t_0+T} +d(x_0, y)  \right)^{\frac{\a}{2}}e^{-\left( \frac{1}{4+3\d} - \frac{1+\tilde{\eta}}{2(4+\e)}  \right)\frac{d^2(x_0,y)}{t_0+T}}dy \\
\leq &  \frac{\widetilde C_\d^{\frac{d}{2}}(t_0+T)^{\frac{d}{2}}C\cdot C_{\e,\tilde{\eta}}}{\left| B(x_0, \sqrt{t_0+T}) \right|} \left\{ \int C e^{-\left( \frac{1}{4+3\d} - \frac{1+\tilde{\eta}}{2(4+\e)} \right)\frac{d^2(x_0,y)}{t_0+T}}dy \right. \\
& \left.  + C \left( \sqrt{t_0+T} \right)^{\frac{\a}{2}} \int \left( 1+ \left( \frac{d(x_0,y)}{\sqrt{t_0+T}} \right)^{\frac{\a}{2}} \right)e^{-\left( \frac{1}{4+3\d} -\frac{1+\tilde{\eta}}{2(4+\e)} \right)\frac{d^2(x_0,y)}{t_0+T}}dy \right\}.
\endaligned
\end{equation}
 Suppose that there exists $\tilde\gamma>0$ such that $\frac{1}{4+3\d} -\frac{1+\tilde{\eta}}{2(4+\e)} > \tilde\gamma$. Then by a similar argument as we derive (\ref{eqn-HA34}),  we deduce
 \begin{equation}\label{eqn-HHA3}\aligned
\widetilde I_1 \leq \frac{C(\d,\e,\tilde{\eta},d)(t_0+T)^{\frac{d}{2}}\left( 1+(t_0+T)^{\frac{\a}{4}} \right)}{\left| B(x_0,  \sqrt{t_0+T}) \right|} \int_M e^{- \tilde\gamma \frac{d^2(x_0,y)}{t_0+T}}dy.
\endaligned
\end{equation}
 If
\[
\widetilde\e_0:= \frac{ \tilde\gamma}{4} - \frac{1}{4+\e} > 0,
\]
  then following  a similar  approach in obtaining  (\ref{eqn-HA40}), we can conclude that
 \begin{equation}\label{eqn-HHA4}\aligned
& \frac{1}{\left| B(x_0, \sqrt{t_0+T}) \right|}\int_M e^{- \tilde\gamma \cdot \frac{d^2(x_0,y)}{t_0+T}}dy \\
\leq & \sum_{i=1}^\infty C e^{-\widetilde\e_0 \cdot 2^{2i}} \cdot 2^{\a i} (t_0+T)^{\frac{\a}{2}}  + \sum_{i=1}^\infty C e^{- \widetilde\e_0\cdot 2^{2i}} +1 \\
 \leq & C_1(t_0+T)^{\frac{\a}{2}} +C_2.
\endaligned
\end{equation}
 As before, we require that there exist positive constants $\tilde\eta$, $\tilde\gamma$ and $\widetilde\e_0$ such that
\begin{equation*}
 \begin{cases}
    \frac{1}{4+3\d} -\frac{1+\tilde{\eta}}{2(4+\e)} > \tilde\gamma>0,\\
\widetilde\e_0= \frac{ \tilde\gamma}{4} - \frac{1}{4+\e} > 0,\\
    \tilde\eta>0.
 \end{cases}
\end{equation*}
This can be achieved if
\begin{equation*}
    \frac{1}{4+3\d} -\frac{1}{2(4+\e)} > \frac{4}{4+\e},
\end{equation*}
which is equivalent to that
\begin{equation}\label{e-d-e-3}
    \e>\frac{27}{2}\d+14.
\end{equation}

Substituting (\ref{eqn-HHA4}) into (\ref{eqn-HHA3}),  we obtain
  \begin{equation}\label{eqn-HHA5}\aligned
\widetilde I_1 \leq C(\d, \e, d)(t_0+T)^{\frac{d}{2}} \left( 1+(t_0+T)^{\frac{3\a}{4}} \right).
\endaligned
\end{equation}
Let
  \begin{equation*}\aligned
\widetilde I_2 :=  \frac{1}{\sqrt{\left| B(x_0, \sqrt{t_0+T}) \right|}} \int \frac{e^{-\frac{d^2(x_0,y)}{(4+2\d)(t_0+T)}}}{\sqrt{\left| B(y, \sqrt{t_0+T}) \right|}}dy.
\endaligned
\end{equation*}
then by a similar proof to get (\ref{eqn-HA43}),  we can deduce
  \begin{equation}\label{eqn-HHA5b}\aligned
\widetilde I_2 \leq C(\d, \e) \left( 1+(t_0+T)^{\frac{3\a}{4}} \right).
\endaligned
\end{equation}
Combining  (\ref{eqn-HA48}), (\ref{eqn-HHA5}) with (\ref{eqn-HHA5b}), it follows
  \begin{equation}\label{eqn-HHA6}\aligned
 I  \leq & C(\d, \e, d) (t_0+T)^{\frac{d}{2}}\left( 1+(t_0+T)^{\frac{3\a}{4}} \right) \\
&+ C(\d, \e, d) \left( d^d(x_0, 0) + ( T+1)^{\frac{d}{2}} \right)\left( 1+(t_0+T)^{\frac{3\a}{4}} \right).
\endaligned
\end{equation}
Substituting (\ref{eqn-HHA6}) into (\ref{eqn-HA46}),  we derive
\begin{equation}\label{eqn-HHA7}\aligned
\left| \p_{t_0}^k u(x_0,t_0) \right| \leq \frac{C(\d,\e,d,k)\left( 1+(t_0+T)^{2m+\frac{5}{2}\a} \right)\left( d^d(x_0, 0) + (t_0+T)^{\frac{d}{2}}+ (T+1)^{\frac{d}{2}} \right)}{\left(t_0+T\right)^k}.
\endaligned
\end{equation}
If $k> 2m+\frac{5}{2}\a+\frac{d}{2}$, we can let $T\to \infty$ to deduce $\p_{t_0}^k u(x_0, t_0)=0$.  Since $(x_0,t_0)$ is arbitrary,  we conclude that $u$ is a polynomial of $t$.  For $t_0 \ge 0$, as explained at the beginning of Step 2, we also have $\p_{t_0}^k u(x_0, t_0)=0$ by uniqueness. Alternatively we can also repeat Step 3 for $t_0>0$ using \eqref {eqn-HAC1} in Step 2. Hence $u$ is a polynomial of time $t$.
\qed

\bigskip

\noindent{\bf Proof of Theorem \ref{thm-2} (b).}
Let $u$ be as in part (a). Once we know it is a polynomial of the time variable $t$, as shown in \cite{LZ19},
\begin{equation}
\label{u=u0+uk}
u(x, t)=u_0(x) + u_1(x) t +... + u_{k}(x) t^{k},
\end{equation}  with
\begin{equation}
\label{equii1}
\Delta u_{i}(x) = (i+1) u_{i+1},
\end{equation} $i=0, ..., k-1$ and $\Delta u_k=0$. By the modified mean value inequality in Section 2, it is easy to see from \eqref{eqn-HA19} that $\partial^k_t u$, as a solution of the heat equation,  also has polynomial growth. Hence $u_k=  \, \partial^k_t u/k! $  has polynomial growth of a fixed degree. Note that, for $i=1, ..., k-1$, we also have
 \[
u_{k-1}= [(k-1)!]^{-1} \, \partial^{k-1}_t u - k t u_k .
 \]So $u_{k-1}$  has polynomial growth of a fixed degree. By induction we know that $u_i$ has polynomial growth of a fixed degree too for $i=0, 1, ..., k$. See also an argument in \cite{CM21} using the Vandermonde matrix. Using equation \eqref{equii1}, we know $u_i$ is the sum of one solution of \eqref{equii1} and harmonic functions, i.e.
 \[
 u_i=f_i + h_i,
 \]where $\Delta h_i=0$ and $\Delta f_i=(i+1) u_{i+1}$. Since $h_i$ is an arbitrary harmonic function and $f_i$ is fixed, we know that $f_i$ has the same order of bound as $u_i$ by choosing $h_i=0$. So $h_i$ has the same order of bound as $u_i$.
   So, if  the space of harmonic functions of polynomial growth of fixed degree has finite dimension, then so does the space of caloric functions. The converse is obvious. \qed

 To finish the section, we give an estimate on the volume of geodesic balls under the conditions of the theorem.

 By the assumption,
\begin{equation*}
\left| B(x, 2r) \right| \leq C(1+r)^\a \left| B(x, r) \right|
\end{equation*}
for any $x\in M$.  Thus we have, for $l=1, 2, 3, ...,$
\begin{equation*}\aligned
\left| B(x, 2^l r) \right| \leq & C(1+2^{l-1}r)^\a (1+2^{l-2}r)^\a \cdot \cdot\cdot (1+r)^\a \cdot  \left| B(x, r) \right|  \\
\leq & C \left( 1+ 2^{\frac{l(l+2)}{2}\a} r^{l\a} \right) \cdot \left| B(x, r) \right|.
\endaligned
\end{equation*}
Let $l$ be the least integer such that $R \le  2^l r$, then $l= [\log_2 \frac{R}{r}]$.
From the above inequality,  we deduce
\begin{equation}\label{eqn-L5b}\aligned
\left| B(x, R) \right| \leq & C\left( 1+ \left( \frac{R}{r} \right)^{\frac{1}{2}\a \left( \log_2 \frac{R}{r} -1 \right)}\cdot r^{\a\log_2 \frac{R}{r}} \right)\cdot \left| B(x, r) \right|.
\endaligned
\end{equation}
Taking $r=1$, we see that the volume of a ball of radius $R$ is allowed to grow like $R^{c \ln R}$ rate.

\vskip24pt

\section{Space of harmonic functions}

 As one application,
 in this section we revisit the question of the finite dimensionality of the space of harmonic functions of polynomial growth on open manifolds.

One key in the   widely used proof of finite dimensionality is the lower bound of the $L^2$ norm, in a smaller ball,  of finite sets of $L^2$ orthonormal harmonic functions in a bigger ball. The lower bound is a multiple of the cardinality of the set and some  constants.   See Lemma 2 in \cite{Li97}  (\cite{Lib} Lemma 28.3) and \cite{CM97}, Proposition 4.16 and more specifically statement (4.20) in the proof e.g.
For later use, we also make a  small generalization by cutting off a fixed ball in the domains of integrals. Although proof can be done in the same way, we present a somewhat more direct proof for completeness and for possible future applications.  This surgery procedure coupled with a quasi monotonicity argument will allow us to do away with local conditions on the manifold and just impose conditions on the ends of the manifold. This results in significant relaxation beyond doubling or mean value inequality, addressing a long time open question or plan. See Theorem \ref{thm-ys} and corollaries below. It is well known that non-trivial local topology  can have a big influence on the heat kernels and hence on the validity of mean value inequality.

\begin{pro}
\label{prcmli}
Let $(M, g)$ be an $n$-dimensional Riemannian manifold satisfying the volume condition: for a constant $\mu>0$,
\begin{equation*}
\left| B(p, r) \right| \leq C r^\mu
\end{equation*}
for some $p \in M$  and $r$ large. Let $K$ be any $k$-dimensional subspace of
\begin{equation*}
\mathcal{H}^d(M):= \{ u \, | \, \D u = 0,  |u(x)| \leq C(1+d(p,  x))^d \},
\end{equation*}
where $p\in M$ is a fixed point. For  any $\beta > 1, \d>0, R_0 \geq 1$, there exists $R>R_0$,  which may depend on $k$ and $R_0$,  such that if $\{ u_i \}_{i=1}^k$ is an orthonormal basis of $K$ w.r.t. the inner product
\begin{equation*}
A_{\beta R}(u, v)= \la u, v \ra_{\beta R} = \int_{B(p, \b R) \backslash B(p, R_0)} u v dx,
\end{equation*}
then
\begin{equation}\label{eqn-L1-2}
\sum_{i=1}^k \int_{B(p, R) \backslash B(p, R_0)} u_i^{2} dx \geq k \b^{-(2d+ \mu+ \d) }.
\end{equation}
\end{pro}

\proof  For a large positive integer $J$, let us write $D_J=B(p, \b^J) \backslash B(p, R_0)$ and introduce the $k \times k$ positive definite matrix
\[
M_J = (\la u_i, u_j \ra_{\beta^J}) \equiv \left( \int_{D_J} u_i u_j dx \right).
\]By the growth condition on $u_i$ and geodesic balls, we have
\begin{equation}
\label{detmj<}
 \det M_J  \le C \beta ^{k J (2d+\mu)} k!.
\end{equation} We argue that for any positive integer $k$, there exists $J=J(k, \beta, \delta, \mu, d)$ such that
\begin{equation}
\label{detmj+1<}
 \det M_{J+1} \le  \det M_J \, \beta ^{k  (2d+\mu+\delta)}.
\end{equation} If this statement were not true, then there exists one integer $k$ such that for all $J=1, 2, 3, ...$, the following holds
\[
 \det M_{J+1} >  \det M_J \, \beta^{k  (2d+\mu+\delta)}.
\]Putting together, these would imply, for all $J$ that
\[
\det M_{J} >  \det M_1 \, \beta^{k (J-1) (2d+\mu+\delta)}.
\]
But this contradicts with \eqref{detmj<} when $J$ is large. Hence \eqref{detmj+1<} is true. Now we can normalize to assume that $\{ u_i \}$ are orthonormal vectors under the inner product $\la \cdot, \cdot \ra_{\beta^{J+1} }$ so that $\det M_{J+1}=1$. Applying the arithmetic-geometric mean inequality on \eqref{detmj+1<}, we obtain
\[
k \le \text{trace} M_J \beta^{2d+\mu+\delta}.
\]The proof is done by renaming $\beta^J $ as $R$.
\qed

The dependence of the radius $R$ on $k$ was not written out in the above cited references. We will show that this dependence is necessary.
In the next two propositions, we will show that the following stronger assertion, which would drastically extend the finite dimensionality result,  is false in general. \\

\noindent{\it {\bf Assertion 4.1}
Let $(M, g)$ be an $n$-dimensional Riemannian manifold satisfying the volume condition: for a constant $\mu>0$,
\begin{equation*}
\left| B(p, r) \right| \leq C r^\mu
\end{equation*}
for some $p \in M$  and $r$ large. Let $K$ be any $k$-dimensional subspace of
\begin{equation*}
\mathcal{H}^d(M):= \{ u \, | \, \D u = 0,  |u(x)| \leq C(1+d(p,  x))^d \},
\end{equation*}
where $p\in M$ is a fixed point. For  any $\beta > 1, \d>0, R_0 \geq 1$, there exists $R>R_0$,  which is independent of $k$,  such that if $\{ u_i \}_{i=1}^k$ is an orthonormal basis of $K$ w.r.t. the inner product
\begin{equation*}
A_{\beta R}(u, v) = \int_{B(p, \b R)} uv dx,
\end{equation*}
then
\begin{equation}\label{eqn-L1}
\sum_{i=1}^k \int_{B(p, R)} u_i^{2} dx \geq k \b^{-(2d+ \mu+ \d) }.
\end{equation}

\begin{pro}
\label{pr4.1}

Let $(M, g)$ be a complete Riemannian manifold  satisfying,  for all harmonic functions $u$,  all $p \in M$ and any $R>0$, that
\begin{equation}\label{eqn-L9}\aligned
|u(p)|^2 \leq C \frac{(1+R)^{m_0}}{\left| B(p, R)\right|}\int_{B(p,R)}|u|^2 dx
\endaligned
\end{equation}
and
\begin{equation}\label{eqn-L10}\aligned
\left| B(p, 2R) \right| \leq C(1+R)^\a \left| B(p, R) \right|.
\endaligned
\end{equation} Here $m_0, \alpha \ge 0$ are any fixed numbers.
Let $K$ be any $k$-dimensional subspace of $\mathcal{H}^d(M)$. Let $\{u_i\}_{i=1}^k$ be a basis for $K$.  For $p\in M,  R\geq 1,  \e\in (0, \beta)$,  we have
\begin{equation}\label{eqn-L11}\aligned
& \sum_{i=1}^k \int_{B(p,R)} u_{i}^2(x)dx  \\
\leq &   C (1+R)^{m_0} \left( 1+ \left( \frac{2}{\e} \right)^{\frac{\a}{2}\left( \log_2 \frac{2}{\e}-1 \right)}\left((1+\e)R\right)^{\a\log_2 \frac{2}{\e}}\right) \sup_{u\in \la A, U \ra}\int_{B(p, (1+\e)R)}u^2 dx.
\endaligned
\end{equation}
Here $\la A, U \ra := \left\{ \left.  v=\sum_i a_i u_i \right|\sum_i a_{i}^2 = 1  \right\}$.

\end{pro}

\begin{proof}

The point is that the right hand side does not have $k$. This is a generalization of known results in \cite{Li97} and \cite{CM98} beyond volume doubling, which may be of independent interest.

Fix $x\in B(p, R)$,  set $K_x:= \left\{ u\in K | u(x) = 0 \right\}$.  The subspace $K_x \subset K$ is at most codimension 1, since for any  $v, w \notin K_x$, $v-\frac{v(x)}{w(x)}w \in K_x$.  Therefore there exists an orthonormal transformation mapping $\{u_i\}_{i=1}^k$ to $\{v_i\}_{i=1}^k$, where $v_i \in K_x, i\geq 2$.

By the assumption (\ref{eqn-L9}),
\begin{equation}\label{eqn-L12}\aligned
\sum_{i=1}^k u_{i}^2(x) = \sum_{i=1}^k v_{i}^2(x) = v_1^{2}(x) \leq C(1+R)^{m_0} \fint_{B(x, (1+\e)R-r(x))} v_1^{2} dy.
\endaligned
\end{equation}  Here and later $r(x)=d(p, x)$.
For any $y \in B(p, R)$, $d(p, y)\leq R$,  then by the triangle inequality,
\[
d(y,x)\leq d(p,y)+d(p,x)\leq R +r(x) \leq 2R.
\]
This implies $y\in B(x, 2R)$.  Thus $B(p,R)\subset B(x, 2R)$ and
\begin{equation}\label{eqn-L13}\aligned
\left| B(p,R) \right| \leq  \left| B(x, 2R) \right|.
\endaligned
\end{equation}
By (\ref{eqn-L5b}),  we obtain
\begin{equation}\label{eqn-L14}\aligned
\left| B(x, 2R) \right| \leq & C\left( 1+ \left( \frac{2R}{(1+\e)R-r(x)}\right)^{\frac{\a}{2}\left( \log_2 \frac{2R}{(1+\e)R-r(x)}-1 \right)} \right.\\
&\left.  \cdot \left(( 1+\e)R-r(x) \right)^{\a \log_2 \frac{2R}{(1+\e)R-r(x)}} \right) \cdot \left| B(x, (1+\e)R-r(x)) \right|.
\endaligned
\end{equation}
Notice  $r(x)\leq R$,  then $(1+\e)R -r(x) \geq \e R$,  that is,
\[
\frac{2R}{(1+\e)R-r(x)} \leq \frac{2}{\e}.
\]
From (\ref{eqn-L13}) and (\ref{eqn-L14}), we derive
\begin{equation}\label{eqn-L15}\aligned
\left| B(p, R) \right| \leq C\left( 1+ \left( \frac{2}{\e} \right)^{\frac{\a}{2}\left( \log_2 \frac{2}{\e}-1 \right)}((1+\e)R)^{\a\log_2 \frac{2}{\e}} \right) \cdot \left| B(x, (1+\e)R-r(x)) \right|.
\endaligned
\end{equation}
Namely,
\begin{equation}\label{eqn-L16}\aligned
\left| B(x, (1+\e)R-r(x)) \right|^{-1} \leq C \left( 1+ \left( \frac{2}{\e} \right)^{\frac{\a}{2}\left( \log_2 \frac{2}{\e}-1 \right)}((1+\e)R)^{\a\log_2 \frac{2}{\e}} \right) \cdot \left| B(p, R) \right|^{-1}.
\endaligned
\end{equation}
Notice that
\begin{equation}\label{eqn-L17}
 B(x, (1+\e)R-r(x)) \subset B(p, (1+\e)R).
\end{equation}
Hence by (\ref{eqn-L12}),  (\ref{eqn-L16}) and (\ref{eqn-L17}),   we can integrate to arrive at
 \begin{equation}\label{eqn-L18}\aligned
& \sum_{i=1}^k \int_{B(p, R)} u_{i}^2(x)dx \\
\leq & C \int_{B(p, R)} \left| B(x, (1+\e)R-r(x)) \right|^{-1} (1+R)^{m_0}  \sup_{u\in \la A, U \ra}\int_{B(p, (1+\e)R)}u^2 dy dx\\
\leq & C (1+R)^{m_0} \left( 1+ \left( \frac{2}{\e} \right)^{\frac{\a}{2}\left( \log_2 \frac{2}{\e}-1 \right)}\left((1+\e)R\right)^{\a\log_2 \frac{2}{\e}}\right) \sup_{u\in \la A, U \ra}\int_{B(p, (1+\e)R)}u^2 dy.
\endaligned
\end{equation}
\end{proof}

\begin{pro}
\label{pro4.2} Assertion 4.1 is false in general.
\end{pro}

\begin{proof}

Let us consider the manifold $(M, g)$ which is topological $\mathbb{R}^2$ equipped with the following metric in geodesic polar coordinates:
\[
g=dr^2 + f^2( r ) d\theta^{2} ,
\] where
\begin{equation}
\label{fr}
f=f(r)=
\begin{cases}
\, r \, \ln^2 r, \qquad r >e,\\
\, \text{smooth}, \qquad 1<r \leq  e,\\
\, r, \qquad 0 \le r \le 1.
\end{cases}
\end{equation} It is easy to check that for $r>e$,  the radial curvature is
\[
-\frac{f''(r)}{f(r)} = -\frac{2}{r^2 \ln r}-\frac{2}{r^2 \ln^2 r} \le -\frac{2}{r^2 \ln r}.
\]Hence by Theorem 3.6 in \cite{Cho84}, the Dirichlet problem at infinite is solvable, consequently
\begin{equation}
\label{hdinfty}
{\rm dim } \mathcal{H} ^d(M)=\infty.
\end{equation}

Now let us check that $(M, g)$ satisfies the assumptions in the assertion and Proposition \ref{pr4.1}.
First we note, for any small $\epsilon_0>0$,
\begin{equation}
\label{bxrln}
|B(x, r)| \le C r^2 ln^2 r \le C r^{2+\epsilon_0}, \quad r> e , \, x \in M.
\end{equation} This is checked as follows. If $r>d(x, 0)/2$, then
\[
|B(x, r)| \le |B(0, 4r)|= 2 \pi \int^{4r}_0 f(r) dr \le c r^2 \ln^2 r.
\]If $r \le d(x, 0)/2$, then in $B(x, r)$ the curvature  is bounded from below by $-K \equiv -c/(r^2 \ln r)$ and the Bishop-Gromov volume comparison theorem implies, for large $r$,
\[
\begin{aligned}
|B(x, r)| &\le |B_{\mathbb{H}^2_K}(r)|=\frac{2\pi}{\sqrt{K}} \int^r_0 \sinh (\sqrt{K} t) dt
 =\frac{2 \pi}{2 K}(e^{\sqrt{K} r} + e^{-\sqrt{K} r} -2)\\
&= c^{-1} r^2 \ln r \left ( e^{\sqrt{c}/(\sqrt{\ln r})} + e^{-\sqrt{c}/(\sqrt{\ln r})}  - 2 \right) \le C r^2.
\end{aligned}
\]

Second, the curvature is negative outside the fixed ball $B(0, e)$, by volume comparison theorem
\begin{equation}
\label{vol>r2}
|B(x, r)|> c r^2.
\end{equation}  To verify this, one can just consider similarly  three cases as above: $ 2e< r \le d(x, 0)/2$,  $d(x, 0)/2<r \le 2 d(x, 0)$ and $r> 2d(x, 0)> 8e$. In the first case one apply volume comparison directly since the curvature is negative there.  In the second case we have $B(x, (d(x, 0)/2)-e) \subset B(x, r)$ and hence
\[
|B(x, r)| \ge |B(x, (d(x, 0)/2)-e)| \ge c (d(x, 0)/2)-e)^2 \ge 0.01 c \, r^2.
\]Here we just applied volume comparison on the ball $B(x, (d(x, 0)/2)-e)$ where the curvature is negative. In the third case we have
$B(x, r) \supset B(0, r/2)$ and we can use the metric to compute $|B(0, r/2)| \ge c r^2 \ln^2 r$.

Therefore, we have polynomial growth for volume and weighted volume doubling:
\[
|B(x, 2r)| \le C (1+r^{\epsilon_0}) |B(x, r)|.
\]Note that this manifold is not volume doubling for some off center balls. One can take
$r=d(x, 0)>>1$  and consider $B(x, r/2)$. By volume comparison, as computed above, due to the faster than quadratic decay of the curvature,  $|B(x, r/2)|$ is comparable to $r^2$.  But one computes directly $|B(x, 4 r)| > |B(0, r)|$ which is like $r^2 \ln^2 r$, and
hence the manifold is not volume doubling.

 The injectivity radius is bounded from below by a positive number since we have bounded curvature and non-collapsed metric. This and the volume conditions \eqref{vol>r2} and \eqref{bxrln} allow us to use Theorem 1.1 in \cite{BCG01}, which says the heat kernel
satisfies
\[
G(x, t, y) \le \frac{C}{t^q} e^{- c d^2(x, y)} = \frac{C t^{2+\epsilon_0-q} }{t^{2+\epsilon_0}} e^{- c d^2(x, y)}
\le \frac{C t^{2+\epsilon_0-q} }{\sqrt{|B(x, \sqrt{t})|} \, \sqrt{|B(y, \sqrt{t})|} } e^{- c d^2(x, y)}
\]for some $q>0$ and all large $t$. Now we can use Theorem \ref{thmmmvi} to conclude that the modified mean value inequality \eqref{eqn-L9} is true.  Therefore Proposition \ref{pr4.1} can be applied since assumptions (\ref{eqn-L9}) and (\ref{eqn-L10}) are met.

Suppose Assertion 4.1 is true. Then by
 Proposition \ref{pr4.1},   we have
\[
\dim \mathcal{H}^d(M) < \infty.
\]
Indeed, let $k$ be the dimension of any finite dimensional subspace $K$ of $\mathcal{H}^d(M)$ and $\{ u_i \}_{i=1}^k$ be an orthonormal basis of $K$ w.r.t. $A_{\b R}$ given in the assertion. We have, by the assertion and the previous proposition with $\e=\beta-1$, that
\begin{equation*}\label{eqn-L19}\aligned
& k \b ^{-(2d+\mu+\delta)} \leq  \sum_{i=1}^k \int_{B(p, R)} u_i^{2} dx  \\
\leq & C (1+R)^{m_0} \left( 1+ \left( \frac{2}{\beta-1} \right)^{\frac{\a}{2}\left( \log_2 \frac{2}{\beta-1}-1 \right)}\left(\beta R\right)^{\a\log_2 \frac{2}{\beta-1}}\right) \sup_{u\in \la A, U \ra}\int_{B(p, \beta R)} u^2 dx.
\endaligned
\end{equation*} Here $\mu=2+\epsilon_0$ and $\alpha=\epsilon_0$.
Since $\{u_i\}_{i=1}^k$ is orthonormal in $A_{\b R}$, then we obtain
\[
\sup_{u\in \la A, U \ra}\int_{B(p, \beta R)}u^2 dx \leq 1.
\]
Hence we have
\[
k\leq C (1+R)^{m_0} \left( 1+ \left( \frac{2}{\beta-1} \right)^{\frac{\a}{2}\left( \log_2 \frac{2}{\beta-1}-1 \right)}\left(\beta R\right)^{\a\log_2 \frac{2}{\beta-1}}\right)\cdot  \b R^{2d+\mu+\d } < \infty.
\]

But this contradicts \eqref{hdinfty}.

\end{proof}

\medskip

 Remark.  One can choose an everywhere  explicit metric in the example in \eqref{fr} so that the curvature is negative everywhere. Take $f=f(r)=r \ln^2(e+r^2)$. Direct calculation shows
\[
f''(r)= \frac{4r^3\ln \left(r^2+ e \right)+12 e r \ln \left(r^2+ e \right)+8r^3}{\left(r^2+ e \right)^2}
\]so that the curvature $K(r)=-f''(r)/f(r)$ is
\[
K(r)= - \frac{4r^2+12 e}{(e+ r^2)^2 \ln \left(r^2+ e \right)} - \frac{8r^2}{\left(r^2+ e \right)^2 \ln^2 \left(r^2+ e \right)}.
\] It is clear that for large $r$, we have
\[
K(r) \le  -\frac{1.5}{r^2 \ln r}
\]and all other argument in the above proposition goes through in the same way.

\medskip

 For this example, we have observed in \eqref{hdinfty} that dim $H^d(M)=\infty$ for all $d \ge 0$.  Since the volume condition of the large geodesic balls in the manifold and the mean value inequality  are just ``off" from the Euclidean ones by no more than $log^2$ factor, this example indicates the obstruction to the
 proposal of relaxing the volume doubling condition or the mean value inequality in \cite{CM98} p118. Rather it shows that the conditions inferring the finite dimensionality results in \cite{CM97}, \cite{Li97} and \cite{CM98} are essentially sharp in the sense that they can not be improved by a large extent. It also shows that $dim \mathbf{K}^d$ in Theorem \ref{thm-2}  can not be finite  in general.

\medskip

 On the other hand, we conclude the paper with a finite dimensionality result beyond the standard condition of volume doubling and mean value inequality. It covers manifolds such as the connected sum of $\mathbb{R}^2$ with a cylinder discussed in the introduction and its higher dimensional analogue such as the connected sum of $\mathbb{R}^4$ with $\mathbb{R}^3 \times \mathbb{S}^1$ proposed in \cite{CM98} p117-118. Roughly speaking the standard conditions are imposed on large dyadic annulus instead of on the whole manifolds, thus significantly extending the coverage of manifolds. For example it covers certain manifolds with multiple ends with non-negative Ricci curvature.  This class of manifolds contains much more variety than those with nonnegative Ricci curvature everywhere and has been studied comprehensively in \cite{LT87} by Li and Tam and other authors. For example, it was proven in \cite{LT87} that the space of positive harmonic manifolds is of finite dimension if the sectional curvature is non-negative outside a compact set. But the case of sign changing harmonic functions is largely open. Quasi monotonicity  of spherical $L^2$ norms of harmonic functions plays a crucial role in addition to the methods in \cite{Li97} and \cite{CM98}. Combining with Theorem \ref{thm-2}, we also obtain additional finite dimensionality results for the space of ancient caloric functions $\mathbf{K}^d$.

\begin{thm}
\label{thm-ys}
Let $(M, g)$ be an $n$-dimensional Riemannian manifold satisfying the following conditions.

(a). polynomial volume bound: for a constant $\mu>0$,
\begin{equation*}
\left| B(p, r) \right| \leq C r^\mu
\end{equation*}
for some $p \in M$  and $r$ large.

(b). For a large positive number $R_0$  the annuli
\[
\mathbb{A}_R \equiv \{ x \, | \, R/2 < d(x, p) < 2 R \}, \qquad R>R_0
\]have at most a fixed number of connected components and the volume conditions
\[
|\mathbb{A}_R \cap E|< c_1 |B(x, R/4)|
\]hold for each connected component $E$ containing $x$;  moreover the mean value inequality  under scale $R/4$ hold, i.e. for $x \in \mathbb{A}_R$ and $ r \in (0, R/4]$, we have, for any harmonic function in $M$,
\[
u^2(x) \le \frac{c_2}{|B(x, r)|} \int_{B(x, r)} u^2(y) dy.
\]

(c). For $r=d(x, p)>R_0 \ge e$ and positive constants $c_3$ and $\epsilon$,  $\Delta r \ge - \frac{c_3}{r (\ln r)^{1+\epsilon} }$ whenever $r$ is smooth, and  the geodesic flow of minimum geodesics from $p$ is onto from $\partial B(p, r_1)$ to $\partial B(p, r_2)$ for all $r_2>r_1>R_0$.

Then for $\mathcal{H}^d(M):= \{ u \, | \, \D u = 0,  |u(x)| \leq C(1+d(p,  x))^d \}$, we have
\[
 \dim  \mathcal{H}^d(M)<\infty.
\]
\end{thm}

\proof  Let $K$ be any $k$ dimensional subspace of $\mathcal{H}^d(M)$. We use Proposition \ref{prcmli} with $\beta=2$. Then for some $R=R(k)>R_0$, we have
\[
k \le 2^{2d+ \mu+ \d } \sum_{i=1}^k \int_{B(p, R) \backslash B(p, R_0)} u_i^{2} dx.
\] Here $\{ u_i \}$ is an orthonormal basis for $K$ in $L^2(B(p, 2 R) \backslash B(p, R_0))$.
 If $u$ is harmonic, it is well known that the spherical integrals $H(r) \equiv \int_{\partial B(p, r)} u^2 dS$ satisfies, a.e.
\[
H'(r)=2 \int_{B(p, r)} |\nabla u|^2 dx + \int_{\partial B(p, r)} u^2 \Delta r dS,
\] where the identity $\partial_r  \ln \sqrt{\det g} = \Delta r$ a.e. in the geodesic polar coordinates is used.   By Condition (c),
\[
H'(r) \ge - \frac{c_3}{r (\ln r)^{1+\epsilon}} H(r), \quad a.e., \quad r>R_0,
\]Therefore
For $r_2>r_1 \ge R_0$, we have
\begin{equation}
\label{hr1<hr2}
H(r_1) \le H(r_2) e^{(c_3/\epsilon)  [(\ln r_1)^{-\epsilon} - (\ln r_2)^{-\epsilon}]} \le H(r_2) e^{c_3/\epsilon}.
\end{equation}

 Due to the possible presence of cut-locus, we need to make the above argument rigorous.
Pick $r_2>r_1>R_0$. We have
\[
H(r_2)-H(r_1)=\int_{\mathbb{S}^{n-1}} \left( u^2(r_2, \phi) \sqrt{\det g(r_2, \phi)}-u^2(r_1, \phi) \sqrt{\det g(r_1, \phi)} \right) d\mathbb{S}^{n-1},
\]where $\phi \in \mathbb{S}^{n-1}$, the unit $n-1$ sphere. This implies
\begin{equation}
\label{h2-h1}
\begin{aligned}
H(r_2)-H(r_1)&=\int_{\mathbb{S}^{n-1}} \left( u^2(r_2, \phi) -u^2(r_1, \phi) \right) \sqrt{\det g(r_2, \phi)}  d\mathbb{S}^{n-1}\\
 &\qquad + \int_{\mathbb{S}^{n-1}} u^2(r_1, \phi) \left( \sqrt{\det g(r_2, \phi)}  -
 \sqrt{\det g(r_1, \phi)} \right) d\mathbb{S}^{n-1}.
\end{aligned}
\end{equation} By Condition (c), in the geodesic polar coordinates, the points $(r_2, \phi),  (r_1, \phi)$ can be joined by a minimum geodesic. Since the cut-locus with respect to the origin is star shaped, only $(r_2, \phi)$ in the geodesic can be in the cut-locus. For $r \in [r_1, r_2)$, we can
use Condition (c) and the relation $\partial_r  \ln \sqrt{\det g} = \Delta r \ge - \frac{c_3}{r (\ln r)^{1+\epsilon} }$ classically to deduce,
\[
\det g(r, \phi) \ge \det g(r_1, \phi) e^{-(c_3/\epsilon)  [(\ln r_1)^{-\epsilon} - (\ln r)^{-\epsilon}]}.
\] Letting $r \to r_2$, we conclude
\[
\det g(r_2, \phi) \ge \det g(r_1, \phi) e^{-(c_3/\epsilon)  [(\ln r_1)^{-\epsilon} - (\ln r_2)^{-\epsilon}]}.
\]
Note due to the quasi-monotonicity, even if $(r_2, \phi)$ is in the cut-locus, we still can define $\det g(r_2, \phi)$ to be the supremum of the limits of $g(r, \cdot)$ along all possible minimum geodesics. Substituting this to the last integral in \eqref{h2-h1}, we see that
\[
H(r_2)-H(r_1) \ge \int_{\mathbb{S}^{n-1}} \left( u^2(r_2, \phi) -u^2(r_1, \phi) \right) \sqrt{\det g(r_2, \phi)}  d\mathbb{S}^{n-1} + e^{-(c_3/\epsilon)  [(\ln r_1)^{-\epsilon} - (\ln r_2)^{-\epsilon}]} H(r_1).
\]This shows
\[
H'(r) \ge 2 \int_{\partial B(p, r)} u \partial_r u dS - \frac{c_3}{r (\ln r)^{1+\epsilon}} H(r) \ge  - \frac{c_3}{r (\ln r)^{1+\epsilon}} H(r),
\] validating \eqref{hr1<hr2}.

Therefore, for some $c_4 \ge 1$,
\begin{equation}
\label{u2r0.75}
\int_{B(p, R)\backslash B(p, R_0)} u_i^{2} dx \le c_4 \int_{B(p, R)\backslash B(p, 0.75 R)} u_i^{2} dx
\end{equation} and hence
\begin{equation}
\label{k<}
k \le c_4 2^{2d+ \mu+ \d } \sum_{i=1}^k \int_{B(p, R)\backslash B(p, 0.75 R)} u_i^{2} dx.
\end{equation}
Since $B(p, R)\backslash B(p, 0.75 R)$ is a proper sub-domain of the annulus $\mathbb{A}_R$. We can use the method of rotation and mean value inequality as in Proposition \ref{pr4.1} with $m_0=\alpha=0$, which originally appeared in \cite{Li97} and \cite{CM98}.

Fix $x \in B(p, R)\backslash B(p, 0.75 R)$,  as before,  $K_x:= \left\{ u\in K | u(x) = 0 \right\}$  is at most codimension 1.  Therefore there exists an orthonormal transformation mapping $\{u_i\}_{i=1}^k$ to $\{v_i\}_{i=1}^k$, where $v_i \in K_x, i\geq 2$.

By condition (b), we can use the mean value inequality with radius $r=R/4$ to deduce
\begin{equation}\label{eqn-L122}\aligned
\sum_{i=1}^k u_{i}^2(x) = \sum_{i=1}^k v_{i}^2(x) = v_1^{2}(x) \leq c_2 \fint_{B(x, R/4)} v_1^{2} dy.
\endaligned
\end{equation}
By the triangle inequality, for any $x \in B(p, R)\backslash B(p, 0.75 R)$, we have
\[
B(x, R/4) \subset \mathbb{A}_R \subset B(p, 2 R).
\]
Therefore we can integrate to deduce
 \begin{equation}\label{eqn-L182}\aligned
& \sum_{i=1}^k \int_{B(p, R)\backslash B(p, 0.75 R)} u_{i}^2(x)dx \\
\leq & C c_2 \int_{B(p, R)\backslash B(p, 0.75 R)} \left| B(x, R/4) \right|^{-1}   \sup_{u\in \la A, U \ra}\int_{B(p, 2 R) \backslash B(p, R_0)} u^2 dy dx\\
  \le & C c_2 \int_{B(p, R)\backslash B(p, 0.75 R)} \left| B(x, R/4) \right|^{-1}  dx \\
  = &C c_2 \sum_{j} \int_{[B(p, R)\backslash B(p, 0.75 R)] \cap E_j} \left| B(x, R/4) \right|^{-1}  dx
  \le C c_1 c_2 \cdot \text{number of ends},
\endaligned
\end{equation}  where $E_j$ are the connected components of the annulus. In the above we have used the fact that $\{ u_i \}$ is an orthonormal basis in $L^2(B(p, 2 R)\backslash B(p, R_0))$ and the volume condition in (b). Substituting this to the right hand side of \eqref{k<}, we finish the proof.
\qed
\medskip

 Concerning Condition (c), we mention the condition $\Delta r \ge 0$  whenever $r=d(x, p) (\ge R_0)$ is smooth holds if the Ricci curvature is non-negative outside a compact set and there is no conjugate point along the ray of infinite length containing $p$ and $x$, which are true in a flat cylinder or Euclidean space e.g.. See the proof in the Corollary below, using the fact that a ray without conjugate point is distance minimizing in a tubular neighborhood. The condition on the geodesic flows is satisfied if the ends of the manifold are Euclidean e.g., since beyond a compact set, geodesics become straight lines. So all conditions of the theorem are satisfied if the ends of the manifold are Euclidean. Therefore, we have proven:

 \begin{cor}

Let $(M, g)$ be a complete Riemannian manifold with finitely many ends which are Euclidean domains, then
\[
\dim \mathcal{H}^d(M)<\infty.
\]
\end{cor}

As mentioned, the point is that no condition is imposed on the compact part of the manifold.
In addition, by examining the proof of Theorem \ref{thm-ys}, we see that the same conclusion holds for the space of harmonic functions on exterior domains with Dirichlet and Neumann boundary condition.

 \begin{cor}

Let $(M, g)$ be a complete Riemannian manifold satisfying the same conditions as Theorem \ref{thm-ys}. Let $D_0 \subset M$ be a compact domain. Let $\mathcal{H}^d(M \backslash D_0)$ be the space of harmonic functions in $M \backslash D_0$ which have  growth rate of degree at most $d$ and which satisfy the Dirichlet or the Neumann boundary condition on $\partial D_0$. Then
\[
\dim \mathcal{H}^d(M \backslash D_0)<\infty.
\]
\end{cor}


\begin{cor}
Let $M$ be a connected sum of $\mathbb{R}^n$ with $\mathbb{R}^{n-1} \times \mathbb{S}^1$, $n \ge 2$,  then
\[
\dim \mathcal{H}^d(M) <\infty.
\]
\end{cor}
\proof Let us pick the reference point $p$ in the smaller end which is a subset of $\mathbb{R}^{n-1} \times \mathbb{S}^1$.
After cutting off a sufficiently large ball $B(p, R_0)$, the annuli essentially become those in a manifold with zero Ricci curvature: $\mathbb{R}^{n}$ or $\mathbb{R}^{n-1} \times S^1$. In the latter case, the universal covers is $\mathbb{R}^n$ (periodic in one variable in PDE language). So the required volume condition and mean value inequality are valid. Now let us verify Condition (c) for the large end.   Let $\gamma$ be a ray starting from the reference point $p$, which goes through the large end $\mathbb{R}^{n}$ eventually.  Pick $x \in \gamma$ which is far away from $p$ so that $r$ is smooth. Note the Jacobi fields are linear in the large end.  Pick  another point $q \in \gamma$ which is even further from $p$. Consider the function $f=f(y) = d(p, y) + d(y, q)$,  $y \in M$. Since $f(x)$ is an interior minimum, we know that $\Delta f(x) \ge 0$, hence
\[
\Delta d(p, x) \ge -\Delta d(x, q) \ge - \frac{n-1}{d(x, q)}.
\] Here we have used the Laplace comparison theorem which is true since the Ricci curvature is non-negative outside a compact set of $M$ and $r=d(p, x)$ and $d(x, q)$  are smooth in the large end. Letting $q$ go to infinity along the ray $\gamma$, we deduce $\Delta d(p, x) \ge 0$.

For the smaller end, the above argument may fail due to the presence of cut locus. In fact it is not true that the exponential map is onto for each pair of spheres.  However by going over the proof of the theorem, it is clear that we just need to prove a slightly modified version of  \eqref{u2r0.75} on the smaller end, namely, for harmonic functions $u$ on $M$ and all large $R>0$,
\begin{equation}
\label{u2r0.75s}
\int_{[B(p, R)\backslash B(p, R_0)] \cap E_s} u^{2} dx \le c_4 \int_{[B(p, 1.1 R)\backslash B(p, 0.75 R)] \cap E_s} u^{2} dx,
\end{equation} where $E_s =\mathbb{R}^{n-1} \times \mathbb{S}^1$. Note the outer radius of the annulus on the righthand side is increased to $1.1 R$. Points $x$ in  $E_s$ can be denoted by a global coordinates $x=(x', \theta)$ where $x' \in \mathbb{R}^{n-1}$ and $\theta \in [0, 2\pi]$. For simplicity, we take the coordinate of the reference point $p$ to be $(0, 0)$, i.e. $x'=0$ and $\theta=0$. When $R>R_0$ and $R_0$ is sufficiently large, for any point $x \in [B(p, R)\backslash B(p, R_0)] \cap E_s$, we know by the triangle inequality that
\[
 |x'|-c_0 \le |x'|-2 \pi- 0.5 c_0 \le  d(x, p) \le |x'|+ 2 \pi + 0.5 c_0 \le |x'| + c_0
\]for a fixed constant $c_0$. Here $|x'|=\sqrt{x^2_1+...+x^2_{n-1}}$. Let us denote by $B'(\rho)$ to be the ball in $\mathbb{R}^{n-1}$ with radius $\rho>0$ and center $0$. Then the above inequalities imply
\[
[B(p, R)\backslash B(p, R_0)] \cap E_s \subset [B'(1.05 R)\backslash B'(0.95 R_0)] \times [0, 2 \pi]
\]which infers
\begin{equation}
\label{bpb'}
\int_{[B(p, R)\backslash B(p, R_0)] \cap E_s} u^{2} dx \le \int_{[B'(1.05 R)\backslash B'(0.95 R_0)] \times [0, 2 \pi]} u^{2} dx.
\end{equation} In the above $dx$ already becomes the Euclidean volume in $E_s$ and $u$ is a harmonic function on $E_s$, which is locally Euclidean.
 Notice that
\[
\begin{aligned}
&\partial_r \int_{\partial B'(r)} u^2 dS= \partial_r \int_{\mathbb{S}^{n-2}} u^2 r^{n-2} d\mathbb{S}^{n-2}=\int_{\mathbb{S}^{n-2}} 2 u \partial_r u r^{n-2} d\mathbb{S}^{n-2}
+\int_{\mathbb{S}^{n-2}} u^2 (n-2) r^{n-3} d\mathbb{S}^{n-2} \\
&\ge
 2 \int_{ B'(r)} u \Delta' u dx' + 2 \int_{ B'(r)} |\nabla' u|^2 dx'.
\end{aligned}
\]Here $\Delta'$ and $\nabla'$ are the Laplace operator and gradient on $\mathbb{R}^{n-1}$ respectively, and $\mathbb{S}^{n-2}$ is the standard unit $n-2$ sphere. Since $0=\Delta u =\Delta' u + \partial^2_\theta u$, the above inequalities imply
\[
\partial_r \int_{\partial B'(r)} u^2 dS \ge - 2 \int_{ B'(r)} u \partial^2_\theta u dx'.
\]This yields, since $u$ is periodic in $\theta$, that
\[
\partial_r \int^{2\pi}_0 \int_{\partial B'(r)} u^2 dS d\theta
\ge - 2 \int^{2\pi}_0 \int_{ B'(r)} u \partial^2_\theta u dx' d\theta =
2 \int^{2\pi}_0 \int_{ B'(r)} |\partial_\theta u|^2 dx' d\theta \ge 0.
\]Substituting this to the righthand side of \eqref{bpb'}, we find, for some constant $C \ge 1$ that
\[
\int_{[B(p, R)\backslash B(p, R_0)] \cap E_s} u^{2} dx \le C \int_{[B'(1.05 R)\backslash B'(0.8 R)] \times [0, 2 \pi]} u^{2} dx \le C \int_{[B(p, 1.1 R)\backslash B(p, 0.75 R)] \cap E_s} u^{2} dx.
\]Here we have used the fact that
\[
[B'(1.05 R)\backslash B'(0.8 R)] \times [0, 2 \pi] \subset [B(p, 1.1 R)\backslash B(p, 0.75 R)] \cap E_s
\]which was elucidated earlier using the triangle inequality. This proves \eqref{u2r0.75s} and the corollary.

\qed

It is clear that the conclusion of the corollary still hold if $M$ is the connected sum of $\mathbb{R}^n$ with finitely many cylinders.

We mention that P. Li in his Theorem 28.7 \cite{Lib} already obtained finite dimensionality results for $\mathcal{H}^d(M)$ under a weak mean value inequality: for some $\beta>1$ and harmonic $u$,
 \[
u^2(x) \le \frac{c_2}{|B(x, r)|} \int_{B(x, \beta r)} u^2(y) dy,
\]
and polynomial volume growth conditions.  But the weak mean value inequality is required on the whole manifold, which is not known to hold on the connected sums.  So his result is different from  Theorem \ref{thm-ys}, where the mean value inequality can also be replaced by the weak mean value inequality with suitable $\beta>1$, but only in a dyadic annulus.

\medskip

\noindent {\bf Acknowledgements}  F. H. L. is partially supported by NSF grants DMS-1955249 and DMS2247773;
 H.Q. is supported by NSFC (No. 12471050) and Hubei Provincial Natural Science Foundation of China (No. 2024AFB746); J.S. is supported by NSFC (Nos. 12071352, 12271039);
 Q.S.Z. gratefully acknowledges the support of Simons'
Foundation grant 710364. He would also like to thank Professors Alexander Grigor\'yan  and G. F. Wei for helpful communications.

\end{document}